\documentclass[letterpaper, 11pt,  reqno]{amsart}
\usepackage{mathtools}
\usepackage{comment}
\usepackage{bbm}
\usepackage{amsmath,amssymb,amscd,amsthm,amsxtra, esint}
\usepackage{amsfonts,latexsym}
\usepackage{eucal}
\usepackage{mathabx}
\usepackage{caption}
\usepackage{subcaption}
\usepackage{mathtools}

\usepackage{lineno}

\usepackage{xcolor}

\usepackage[margin=1.2in,marginparwidth=1.5cm, marginparsep=0.5cm]{geometry}

\usepackage[implicit=true]{hyperref}
\allowdisplaybreaks[2]

\usepackage{color}

\definecolor{gr}{rgb}   {0.,   0.69,   0.23 }
\definecolor{bl}{rgb}   {0.,   0.5,   1. }
\definecolor{mg}{rgb}   {0.85,  0.,    0.85}
\definecolor{yl}{rgb}   {0.8,  0.7,   0.}
\definecolor{or}{rgb}  {0.7,0.2,0.2}

\newtheorem{theorem}{Theorem} [section]

\newtheorem{lemma}[theorem]{Lemma}
\newtheorem{proposition}[theorem]{Proposition}
\newtheorem{remark}[theorem]{Remark}

\newtheorem{corollary}[theorem]{Corollary}

\newtheorem{conjecture}{Conjecture}[section]

\DeclareMathOperator{\med}{med}

\newcommand{\noi}{\noindent}

\newcommand{\R}{\mathbb{R}}
\newcommand{\C}{\mathcal{C}}

\newcommand{\al}{\alpha}
\newcommand{\be}{\beta}

\newcommand{\eps}{\varepsilon}

\newcommand{\cj}{\overline}

\newcommand{\les}{\lesssim}
\newcommand{\ges}{\gtrsim}

\newcommand{\N}{\mathbb{N}}

\usepackage{tikz}

\usetikzlibrary{shapes.misc}
\usetikzlibrary{shapes.symbols}
\usetikzlibrary{shapes.geometric}
\tikzset{
	dot/.style={circle,fill=black,draw=black,inner sep=0pt,minimum size=0.5mm},
	>=stealth,
	}
\tikzset{
	ddot/.style={circle,fill=white,draw=black,inner sep=0pt,minimum size=0.8mm},
	>=stealth,
	}

\tikzset{decision/.style={ 
        draw,
        diamond,
        aspect=1.5
    }}

\tikzset{dia2/.style
={diamond,fill=white,draw=black,inner sep=0pt,minimum size=1mm},
	>=stealth,
	}

\mathtoolsset{showonlyrefs}

\numberwithin{equation}{section}
\numberwithin{theorem}{section}

\usepackage{color}

\newtheorem{thm}{Theorem}[section]
\newtheorem{lem}[thm]{Lemma}
\newtheorem{prop}[thm]{Proposition}

\newcommand{\abrac}[1]{\left\langle #1 \right\rangle}
\theoremstyle{definition}

\theoremstyle{remark}

\def\R{{\mathbb R}}

\def\N{{\mathbb N}}

\def\C{{\mathbb C}}

\newcommand{\del}{\partial}

\newcommand{\Nmax}{N_{\text{max}}}
\newcommand{\Nmed}{N_{\text{med}}}
\newcommand{\Nmean}{N_{\text{mean}}}
\newcommand{\Nmin}{N_{\text{min}}}
\newcommand{\Lmax}{L_{\text{max}}}
\newcommand{\Lmed}{L_{\text{med}}}
\newcommand{\Lmean}{L_{\text{mean}}}
\newcommand{\Lmin}{L_{\text{min}}}

\renewcommand{\Tilde}{\widetilde}

\numberwithin{equation}{section}

\begin{document}
\title{On the well-posedness of the initial value problem for the MMT model}

\author{Mahendra Panthee}
\address{Department of Mathematics, University of Campinas, Brazil}
\email{mpanthee@unicamp.br}
\author{James Patterson}
\address{School of Mathematics, University of Birmingham, UK}
\email{jxp277@student.bham.ac.uk}
\author{Yuzhao Wang}
\address{School of Mathematics, University of Birmingham, UK}
\email{y.wang.14@bham.ac.uk}

\begin{abstract}
This work investigates the initial value problem (IVP) for the two-parameter family of dispersive wave equations known as the Majda-McLaughlin-Tabak (MMT) model, 
which arises in the weak turbulence theory of random waves. 
The MMT model can be viewed as a derivative nonlinear Schrödinger (dNLS) equation where both the nonlinearity and dispersion involve nonlocal fractional derivatives.

The purpose of this study is twofold: first, to establish a sharp well-posedness theory for the MMT model; 
and second, to identify the critical threshold for the derivative in the nonlinearity relative to the dispersive order required to ensure well-posedness. 
As a by-product, we establish sharp well-posedness for non-local fractional dNLS equations; notably, our results resolve the regularity endpoint left open in \cite{WE-22}.
 \end{abstract}
 
 \maketitle
 
 \section{Introduction}
 In this work, we consider the initial value problem (IVP) for the  following two-parameter family of dispersive wave equation, known as the Majda--McLaughlin--Tabak (MMT) model
\begin{equation} \label{MMTEquation}
\begin{cases}
i\del_t u + (-\partial_x^2)^{\frac{\alpha}{2}}u =  \kappa D_x^\beta(|D_x^\beta u|^2 D_x^\beta u ),\\
u(x, 0) = u_0(x),
\end{cases} \qquad x \in \R, t\in\R,
\end{equation}
where $ u= u(x,t)$ is a complex function,  $\alpha>0$, $\kappa \in \{-1,1\}$, $\beta\in \R$ and the operator $D_x^\beta$ is defined via Fourier transform $\widehat{D_x^{\beta}u}(\xi) = |\xi|^{\beta}\widehat{u}(\xi)$. 
The MMT model was originally proposed in \cite{MMT-97} as a simplified Hamiltonian model for one-dimensional wave,
where $\alpha>0$ controls the dispersion relation $\omega(k)=|k|^{\alpha}$ and $\beta\in\R$ modulates the strength and locality of the nonlinearity.  
This two-parameter family interpolates between the classical cubic nonlinear Schrödinger equation (when $\alpha=2$, $\beta=0$) and fractional or nonlocal models relevant to physical settings.  
The model was introduced to test the predictions of weak wave turbulence (WT) theory on a computationally tractable system with a tunable dispersion and nonlinearity~\cite{ZDP-04,ZGPD-01}.
The model \eqref{MMTEquation} is a Hamiltonian system that possesses energy 
\begin{equation}\label{energy}
E(u) := \int \Big(|D_x^{\frac{\alpha}{2}}u|^2 - \frac\kappa2 |D_x^{\beta}u|^4\Big)dx = E(u_0)
\end{equation}
and mass 
\begin{equation}\label{mass}
M(u) := \int |u|^2 dx = M(u_0),
\end{equation}
conservation laws, see \cite{MMT-97}.
These Hamiltonian structures provide a natural $H^{\alpha/2}$ control.

The linear part $(-\partial_x^2)^{\frac\alpha2} u$ generalises the Schrödinger-type dispersion: 
for instance, $\alpha = 2$ reproduces the standard cubic nonlinear Schrödinger scaling, 
while $\alpha = \tfrac{1}{2}$ mimics the deep-water gravity wave dispersion $\omega \sim |k|^{1/2}$.
For $\alpha>1$, the convexity of the dispersion relation suppresses exact four-wave resonances, which tends to stabilise the dynamics and favours local existence; for $\alpha<1$, nontrivial resonances appear, leading to stronger energy transfer and the possibility of wave collapse~\cite{ZGPD-01}.  
These phenomena underlie the ``MMT spectrum'' observed in simulations, which differs from the standard Kolmogorov–Zakharov power laws and reflects the coexistence of weak turbulence with intermittent coherent events~\cite{ZDP-04}.
For $1 < \alpha < 2$, which is the regime of interest in this work, 
the dispersion is supercritical compared to water waves but still subquadratic, 
so that the group velocity remains highly frequency-dependent and the model retains strong dispersive smoothing. 
In particular, unlike the standard NLS case $\alpha = 2$, this intermediate fractional dispersion is nonlocal 
and leads to a delicate balance between dispersive effects and nonlinear resonances; see \cite{CHKL-15}.

The parameter $\beta$ determines whether the cubic interaction is regularising or singular:  
when $\beta=0$, one obtains the standard cubic nonlinearity and the equation behaves similarly to a fractional NLS, for which classical dispersive estimates and Hamiltonian conservation may yield local and global well-posedness in $H^{s}$; see \cite{CHKL-15}.  
Negative $\beta$ introduces additional smoothing operators, weakening nonlinear effects, whereas $\beta>0$ amplifies high-frequency components through derivative-type interactions such as
\[
D^{\beta}u \cdot D^{\beta} \cj u \cdot D^{\beta}u,
\]
creating a derivative-losing regime that is significantly more difficult to control analytically than that of \cite{CHKL-15}.  

In recent years, the Hamiltonian model \eqref{MMTEquation} has drawn considerable attention from the research community; see, for instance, \cite{CMMT-99}, \cite{GLZ-25}, \cite{SW-21}, \cite{ZDP-04}, \cite{ZGPD-01}, and the references therein. In particular, the authors in \cite{GLZ-25} studied the IVP for the kinetic wave equation (the statistical analog of the MMT model) 
\begin{equation*}
\partial_t n(t, k) = \mathcal{C}(n(t))(k),
\end{equation*}
with the collision operator given by
\begin{equation*}
\mathcal{C}(n)(k) = \int_{\R^3}K(k_1,k_2,k_3, k,n_1, n_2, n_3, n_k)\delta(k_1+k_2-k_3-k)\delta(\omega_1+\omega_2-\omega_3-\omega) dk_1dk_2dk_3,
\end{equation*}
where
$$K(k_1,k_2,k_3, k,n_1, n_2, n_3, n_k):=|k_1k_2k_3k|^{2\beta}(n_1n_2n_3+n_1n_2n_k-n_1n_3n_k-n_2n_3n_k),$$
with
\begin{equation*}
\begin{split}
&n_k = n(t, k), \quad n_i= n(t, k_i), \quad i= 1, 2, 3,\\
&\omega = |k|^{\alpha}, \quad \omega_i = |k_i|^{\alpha}, \quad i= 1, 2, 3,
\end{split}
\end{equation*}
and obtained the local well-posedness result.
Their analysis marks the first step towards a rigorous understanding of the kinetic formulation of wave turbulence and indicates that even in one dimension, the MMT family possesses rich analytical structure and subtle regularising effects.

\subsection{Main results}

Well-posedness of IVPs for fractional nonlinear Schrödinger (fNLS) equations has been extensively studied in the literature. For instance, in \cite{HS-15} the authors investigated the fNLS equation introduced in \cite{Las-02} with a generalized nonlinearity and established local well-posedness via Strichartz-type estimates. In \cite{CHKL-15}, the case of cubic nonlinearity was considered, and sharp local well-posedness was obtained in Bourgain spaces. To achieve this, the authors derived trilinear estimates in Bourgain spaces by employing Tao’s multiplier technique from \cite{TT-01}. Regarding global well-posedness in Sobolev spaces of negative regularity, we refer to the recent work \cite{BLLZ-23}. We also highlight several other contributions addressing well-posedness, ill-posedness and stability of fNLS and related models, see for example \cite{CHHO-17}, \cite{CHHO-14}, \cite{CHHO-13}, \cite{CHKL-14}, \cite{CHKL-13}, \cite{CP-18}, \cite{DET-16}, \cite{GH-11}, \cite{IP-14}, \cite{WE-22}, and the references therein.

To the best of our knowledge, there are no existing results on the well-posedness of the IVP \eqref{MMTEquation} in its original form, particularly for the regime $\be > 0$. 
In this work, our aim is to address this gap by studying the IVP \eqref{MMTEquation} for $1<\alpha\leq 2$ with initial data in the classical Sobolev spaces $H^s(\R)$. The MMT model enjoys the scaling property, that is, if $u(x,t)$ is a solution to the IVP \eqref{MMTEquation} then so is $u^{\lambda}(x,t) = \lambda^{\frac{4\beta-\alpha}{2}}u(\frac{x}{\lambda}, \frac{t}{\lambda^{\alpha}})$ for $\lambda>0$. Observe that,
\begin{equation*}
\|u^{\lambda}\|_{\dot{H}^s(\R)} = \lambda^{\frac{4\beta+1-\alpha}{2}-s}\|u\|_{\dot{H}^s(\R)}.
\end{equation*}
Hence, the critical Sobolev regularity index up to which one may expect local well-posedness for the IVP~\eqref{MMTEquation} is given by
\begin{equation*}
s_c := 2\beta + \frac{1 - \alpha}{2}.
\end{equation*}
More precisely, we expect that the equation~\eqref{MMTEquation} is well-posed in the subcritical or critical regime, namely in $H^s(\R)$ for $s \ge s_c$, 
while it becomes ill-posed in the supercritical regime $H^s(\R)$ for $s < s_c$.  
In what follows, we show that in many cases the actual threshold of well-posedness is higher than the formal scaling prediction $s_c$.

\subsection{Well-posedness}\label{SUB:well}

In this section, we establish a series of local well-posedness results for the MMT equation~\eqref{MMTEquation}.  
Our approach is based on deriving a trilinear estimate in the associated Bourgain spaces, employing Tao’s multilinear multiplier method~\cite{TT-01}, and applying the contraction mapping principle.  
We will demonstrate the sharpness of some obtained regularity thresholds in both the Sobolev index and the parameter relations involving~$\alpha$ and~$\beta$.

\subsubsection{Positive regularity case.}
We first consider the case $s \ge 0$.  
By adapting the techniques developed in~\cite{TT-01,CHKL-15}, we establish the key trilinear estimate required for the contraction argument.  
A crucial observation is that, on the hyperplane $\xi_1+\xi_2+\xi_3+\xi_4 = 0$, the inequality
\begin{equation}\label{est-i1}
\langle\xi_4\rangle^s \leq \sum_{j=1}^3 \langle\xi_j\rangle^s, \qquad s \ge 0,
\end{equation}
allows us to reduce the trilinear estimate to a bilinear one.  
This leads to the following result.

\begin{theorem}\label{Th-loc-1} 
Let $1 < \alpha \le 2$ and $-\tfrac{1}{4} < \beta < \tfrac{\alpha - 1}{2}$.  
Then, for any given data $u_0 \in H^s(\R)$, the IVP~\eqref{MMTEquation} is locally well-posed in $H^s (\R)$ for $s \ge \max\left\{ 0,\, 2\beta + \tfrac{2 - \alpha}{4} \right\}$.
\end{theorem}

We remark that the restriction on $\beta$ is sharp in view of Theorem~\ref{Th-ill-Line} below. In particular, if $\beta<-\frac14$ or $\beta>\frac{\al-1}{2}$, then the local well-posedness statement in Theorem~\ref{Th-loc-1} fails for every $s\in\R$.
See Remark \ref{RMK:cr} for further discussion.

\subsubsection{Negative regularity case.}
When $s < 0$, the relation~\eqref{est-i1} is no longer valid.  
Instead, depending on the minimum frequency, one may employ the following inequality:
\begin{equation}\label{est-i2}
\frac{\langle\xi_1\rangle^{-s}\langle\xi_2\rangle^{-s}\langle\xi_3\rangle^{-s}}
{\langle\xi_4\rangle^{-s}} 
\lesssim 
\langle\xi_1\rangle^{-\tfrac{3}{2}s}\langle\xi_3\rangle^{-\tfrac{3}{2}s},
\end{equation}
which, again, removes the $\xi_4$ dependence, thereby reducing the estimate to a bilinear form.  
Proceeding in the same way as for the positive regularity case yields an analogous estimate, but with the loss of an extra factor $\tfrac{s}{2}$ in the regularity exponent.  
This leads to the following result.

\begin{corollary}\label{Th-loc-2} 
Let $1 < \alpha \le 2$ and $-\tfrac{1}{4} < \beta < \tfrac{\alpha - 1}{2}$.  
Then, for any given data $u_0 \in H^s(\R)$, the IVP~\eqref{MMTEquation} is locally well-posed in $H^s (\R)$ for $s < 0$ and
$s \ge \tfrac{4}{3}\beta + \tfrac{2 - \alpha}{6}$.
\end{corollary}

The bound~\eqref{est-i2} is rather crude and yields a non-optimal regularity threshold.  
To improve this, we refine the analysis by distinguishing between two frequency configurations:
\begin{itemize}
\item[(i)] If $|\xi_4| \gtrsim |\xi_j|$ for some $j \in \{1,2,3\}$, we can again reduce to the bilinear setting as in~\eqref{est-i1}.
\item[(ii)] If $|\xi_4| \ll |\xi_j|$ for all $j=1,2,3$, such a reduction is no longer possible; we therefore estimate the full trilinear form directly in Section \ref{sec-4}.
\end{itemize}
The latter case yields an improved regularity threshold for a wide range of $\beta < 0$, leading to the following theorem (but with a more restrictive range of $\al$).

\begin{theorem}\label{Th-loc-3} 
Let $1 < \alpha \le 2$ and $-\tfrac{1}{4} < \beta < \tfrac{\alpha - 1}{2}$.  
Then, for any given data $u_0 \in H^s(\R)$, the IVP~\eqref{MMTEquation} is locally well-posed whenever
$0 > s > \beta + \tfrac{5}{8} - \tfrac{\alpha}{2}$ and $
s \ge 2\beta + \tfrac{2 - \alpha}{4}$.
\end{theorem}

Theorem \ref{Th-loc-3} improves Corollary \ref{Th-loc-2} for most, but not all, ranges of $\alpha$ and $\beta$. 
When $\beta + \tfrac{5}{8} - \tfrac{\alpha}{2} < 2\beta + \tfrac{2 - \alpha}{4}$, the trilinear refinement further improves Theorem~\ref{Th-loc-1} to negative regularities.  
Collecting the results of Theorem~\ref{Th-loc-1}, Corollary~\ref{Th-loc-2} and Theorem~\ref{Th-loc-3}, we obtain the complete picture of local well-posedness established in this work.

\begin{remark}[Summary of local well-posedness thresholds]\label{RMK:LWP}
\rm 
Let $1 < \alpha \le 2$, $-\tfrac{1}{4} < \beta < \tfrac{\alpha - 1}{2}$, and $\delta > 0$ be arbitrary.  
Define
\begin{equation}\label{s-beta-line}
s_{\beta,\alpha} := 
\begin{cases}
\frac{4}{3}\beta + \frac{2 - \alpha}{6}, & \text{for } \beta \in \big(-\tfrac{1}{4},\, \tfrac{7}{8} - \alpha\big], \\[6pt]
\beta + \frac{5}{8} - \frac{\alpha}{2} + \delta, & \text{for } \beta \in \big(\tfrac{7}{8} - \alpha,\, \tfrac{1}{8} - \tfrac{\alpha}{4}\big], \\[6pt]
2\beta + \frac{2 - \alpha}{4}, & \text{for } \beta \in \big(\tfrac{1}{8} - \tfrac{\alpha}{4},\, \tfrac{\alpha - 1}{2}\big).
\end{cases}
\end{equation}
Then, for any initial data $u_0 \in H^s(\R)$, the IVP~\eqref{MMTEquation} is locally well-posed whenever $s \ge s_{\beta,\alpha}$.  
This piecewise function $s_{\beta,\alpha}$ thus represents the best available regularity threshold obtained from the combination of Theorem~\ref{Th-loc-1}, Corollary~\ref{Th-loc-2} and Theorem~\ref{Th-loc-3}.

The different cases and regions of regularity where the  well-posedness result holds are shown in Figure \ref{fig-1} below.\\

\begin{figure}[ht]

\centering
\begin{subfigure}[b]{0.5\textwidth} 
\resizebox{\textwidth}{!}{
\begin{tikzpicture}
  \draw[<->] (-4.2,0) -- (3.5,0) node[right] {$\beta$};
  \draw[<->] (0,-3.0) -- (0,3.5) node[above] {$s$};
  \draw[magenta] (-3.0, -2.5)  -- (2,3)  node[right] {\small{$s =2\beta+\frac{2-\alpha}{4}$}}; 
    \draw[black][dashed] (-3.0,-3) -- (-3.0, 3) node[left]{$\beta = -\frac14$};
    \draw[black][dashed] (2,-3)node[right]{$\beta = \frac{\alpha-1}{2}$} -- (2, 3) ;
\pgfplothandlerpolarcomb
\pgfplotstreamstart
\pgfplotstreampoint{\pgfpoint{-1cm}{0cm}}
  \pgfpathcircle{\pgfpoint{-3cm}{0cm}} {1pt}
  \pgfpathcircle{\pgfpoint{2cm}{0cm}} {1pt}
  \pgfusepath{fill}
\pgfplotstreamend
\pgfusepath{stroke}

\end{tikzpicture}}
\subcaption{\label{fig-a}$\frac32 \leq \alpha \leq 2$ }
\end{subfigure}
~
\hspace{-1cm}
\begin{subfigure}[b]{0.55\textwidth}
\resizebox{\textwidth}{!}{
\begin{tikzpicture}

  \draw[<->] (-4.2,0) -- (3.5,0) node[right] {$\beta$};
  \draw[<->] (0,-3.0) -- (0,3.5) node[above] {$s$};

  \draw[black][dotted] (-1.5,-0.85) -- (-1.5,0) node[above]{$\frac18 - \frac{\alpha}{4}$};
  \draw[magenta] (-3.0, -2.5)  -- (2,3)  node[right] {\small{$s =2\beta+\frac{2-\alpha}{4}$}}; 

  \draw[blue][dashed] (-3, -1.5) node[left]{\small{$s=\beta+\frac58-\frac{\alpha}{2}$}} -- (-1.5,-0.85) ;
 \draw[blue] (-3.8, -0.8) ;

    \draw[black][dashed] (-3.0,-3) -- (-3.0, 3) node[left]{$\beta = -\frac14$};
    \draw[black][dashed] (2,-3)node[right]{$\beta = \frac{\alpha-1}{2}$} -- (2, 3) ;
  
\pgfplothandlerpolarcomb
\pgfplotstreamstart
\pgfplotstreampoint{\pgfpoint{-1cm}{0cm}}
  \pgfpathcircle{\pgfpoint{-1.5cm}{-0.85cm}} {1pt}
  \pgfpathcircle{\pgfpoint{-3cm}{0cm}} {1pt}
  \pgfpathcircle{\pgfpoint{2cm}{0cm}} {1pt}
   \pgfpathcircle{\pgfpoint{-1.5cm}{0cm}} {1pt}
  \pgfusepath{fill}
\pgfplotstreamend
\pgfusepath{stroke}

\end{tikzpicture}}
\subcaption{\label{fig-b}$\frac{9}{8} < \alpha < \frac{3}{2}$}

\end{subfigure}

\begin{subfigure}[b]{0.55\textwidth}
\resizebox{\textwidth}{!}{
\begin{tikzpicture}

  \draw[<->] (-4.2,0) -- (3.5,0) node[right] {$\beta$};
  \draw[<->] (0,-3.0) -- (0,3.5) node[above] {$s$};

  \draw[black][dotted] (-1.5,-0.85) -- (-1.5,0) node[above]{\footnotesize{\hspace{0.15cm}$\frac18 - \frac{\alpha}{4}$}};
  \draw[magenta] (-3.0, -2.5)  -- (2,3)  node[right] {$s =2\beta+\frac{2-\alpha}{4}$}; 

  \draw[blue][dashed] (-2.25, -1.17) node[left]{\hspace{-3.9cm}\small{$s=\beta+\frac58-\frac{\alpha}{2}$}} -- (-1.5,-0.85) ;
 \draw[blue] (-3.8, -0.8) ;

 \draw[orange] (-3, -1.75) node[left]{\small{$s=\frac43\beta+\frac{2-\alpha}{6}$}} -- (-2.25,-1.17) ;
 \draw[blue] (-3.8, -0.8) ;

    \draw[black][dashed] (-3.0,-3) -- (-3.0, 3) node[left]{$\beta = -\frac14$};
    \draw[black][dashed] (2,-3)node[right]{$\beta = \frac{\alpha-1}{2}$} -- (2, 3) ;
  \draw[black][dotted] (-2.25,-1.17) -- (-2.25, 0) node[above]{\hspace{-0.15cm}\footnotesize{$\frac78 - \alpha$}};
    \draw[black][dashed] (2,-3)node[right]{$\beta = \frac{\alpha-1}{2}$} -- (2, 3);
\pgfplothandlerpolarcomb
\pgfplotstreamstart
\pgfplotstreampoint{\pgfpoint{-1cm}{0cm}}
  \pgfpathcircle{\pgfpoint{-1.5cm}{-0.85cm}} {1pt}
  \pgfpathcircle{\pgfpoint{-2.25cm}{-1.17cm}} {1pt}
  \pgfpathcircle{\pgfpoint{-2.25cm}{0cm}} {1pt}
  \pgfpathcircle{\pgfpoint{-3cm}{0cm}} {1pt}
  \pgfpathcircle{\pgfpoint{2cm}{0cm}} {1pt}
   \pgfpathcircle{\pgfpoint{-1.5cm}{0cm}} {1pt}
  \pgfusepath{fill}
\pgfplotstreamend
\pgfusepath{stroke}

\end{tikzpicture}}
\subcaption{\label{fig-c}$1< \alpha \leq \frac{9}{8}$}
\end{subfigure}
\caption{\label{fig-1} Regions for local well-posedness}
\end{figure}
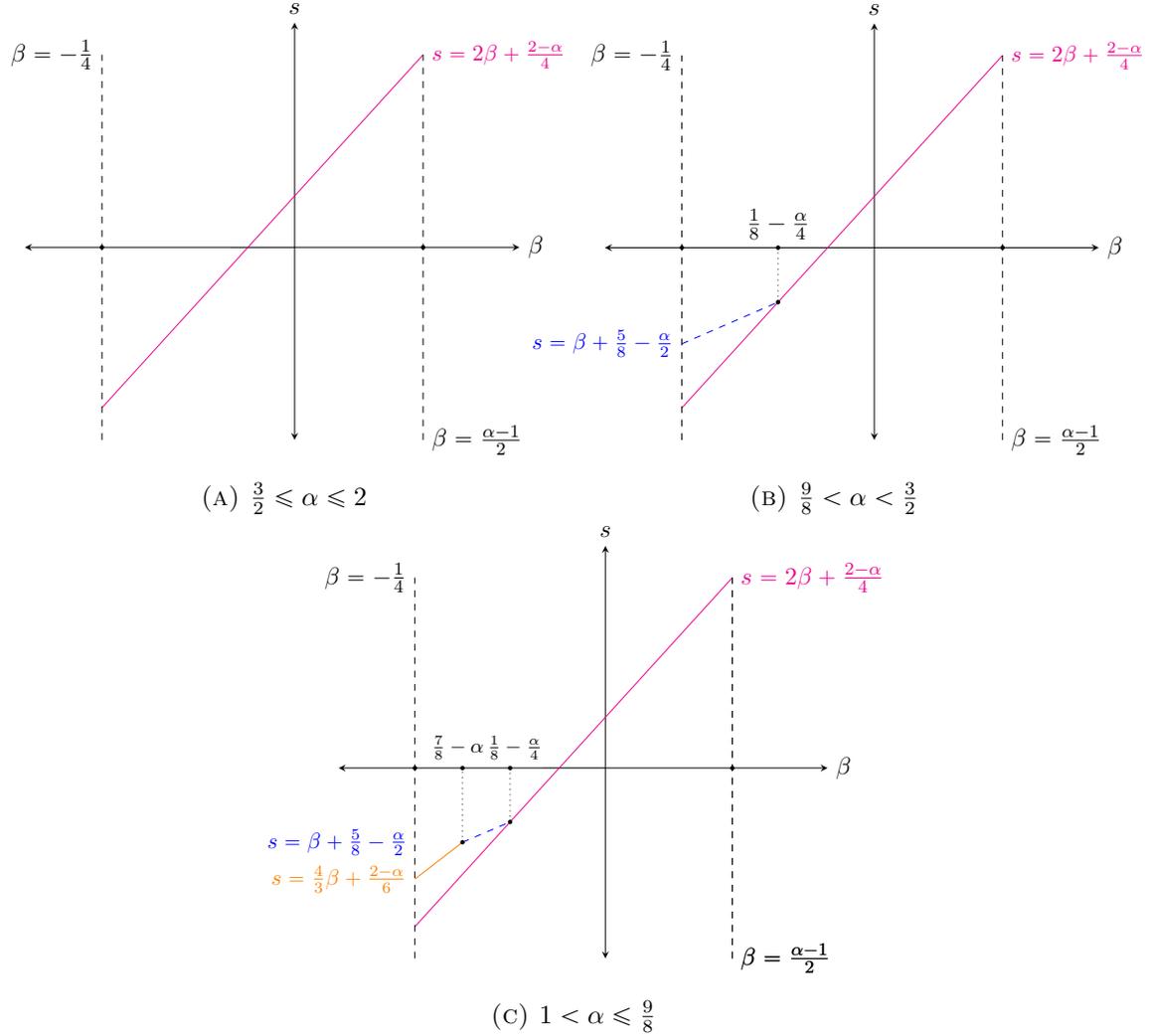

From the figures, we observe that the regions of well-posedness vary with the dispersion parameter~$\alpha$.  
For larger dispersion, namely $\tfrac{3}{2} \le \alpha \le 2$ (Figure~(A)), the local well-posedness holds for all points lying on and above the magenta line 
\[
s = 2\beta + \tfrac{2-\alpha}{4}.
\]
When the dispersion weakens to $\tfrac{9}{8} < \alpha < \tfrac{3}{2}$ (Figure~(B)), 
an additional restriction appears: the well-posedness region lies above both the magenta line 
$s = 2\beta + \tfrac{2-\alpha}{4}$ 
and the dashed blue line 
$s = \beta + \tfrac{5}{8} - \tfrac{\alpha}{2}$, 
with the transition point occurring at $\beta = \tfrac{1}{8} - \tfrac{\alpha}{4}$.
Finally, for the weakest dispersion range $1 < \alpha \le \tfrac{9}{8}$ (Figure~(C)), 
the admissible regularity region is further expanded by the orange curve 
\[
s = \tfrac{4}{3}\beta + \tfrac{2-\alpha}{6},
\]
which dominates for $\beta < \tfrac{7}{8} - \alpha$.
Hence, as $\alpha$ decreases, the boundary of the well-posedness region descends and the dependence on~$\beta$ becomes increasingly nonlinear, 
reflecting the greater difficulty of establishing well-posedness in the weakly dispersive regime.
\end{remark}

\begin{remark} \rm \label{RMK:high}
It is worth emphasising that the well-posedness results established in this work continue to hold even when the dispersion is of higher order. In particular, the analysis extends without essential modification to the regime $\alpha>2$, where the stronger dispersive effect further improves the smoothing and decay properties of the linear flow.
\end{remark}

\begin{remark}\rm 
\label{RMK:cr}
We remark that the restriction on $\beta$ is sharp in view of Theorem \ref{Th-ill-Line} below. 
Specifically, if $\beta < -\frac{1}{4}$ or $\beta > \frac{\alpha-1}{2}$, the local well-posedness statements in Theorem \ref{Th-loc-1}, Corollary \ref{Th-loc-2}, and Theorem \ref{Th-loc-3} fail for every $s \in \mathbb{R}$. 
This demonstrates that for a given dispersion of order $\alpha$, 
there is a fundamental limit on the strength of the nonlinear derivative for which well-posedness can be established.
The primary challenge in solving derivative nonlinear dispersive PDEs is managing the `derivative loss.' 
To address this, the central strategy is to exploit the `smoothing effect' inherent in the dispersion to compensate for the derivatives in the nonlinearity. 
The condition $\beta < \frac{\alpha-1}{2}$ appears to be the sharp criterion for the maximum derivative loss that a dispersive operator of order $\alpha$ can accommodate. 
We expect this to be a general principle for derivative nonlinearities: the smoothing gain from dispersion can only counteract a specific amount of derivative loss in the nonlinearity.
Further discussion and additional examples are provided in Subsection \ref{SUB:dNLS}.
\end{remark}


Having established local well-posedness, a natural question arises: can the local solution be extended globally in time? Using energy and mass conservation laws \eqref{energy} and \eqref{mass} one can obtain an {\em a priori} bound on the solution for data in $H^{\frac{\alpha}{2}}(\R)$ there by obtaining the following global well-posedness result.

\begin{theorem}\label{Th-gwp-Line}
The local solution to the IVP \eqref{MMTEquation} when $\kappa = -1$ given by Theorem \ref{Th-loc-1}, Corollary  \ref{Th-loc-2}, and and Theorem \ref{Th-loc-3} can be extended globally in time for any given data $u_0\in H^{s}(\R)$, $s\geq \frac{\alpha}{2}$.
\end{theorem}

\subsection{Ill-posedness}
The regularity required on the initial data to obtain the local well-posedness results in Theorem \ref{Th-loc-1}, Corollary \ref{Th-loc-2}, and Theorem \ref{Th-loc-3} is considerably higher than that suggested by the scaling argument. A natural question, is whether these results are sharp. To address this question, at least partially, we establish the following ill-posedness result.

 \begin{theorem}\label{Th-ill-Line}
Let $1< \alpha\leq 2$, then for any given data $u_0\in H^s(\R)$, there exists no time $T=T(\|u_0\|_{H^s(\R)})$ such that the application that takes the initial data $u_0$ to the solution $u\in C([0, T]: H^s(\R))$ of the IVP \eqref{MMTEquation} is $C^3$ at the origin if $s< 2\beta + \frac{2-\alpha}{4}$ or $\beta<-\frac14$ or $\beta>\frac{\alpha-1}2$.
\end{theorem}
To prove this result we constructed appropriate counter-examples based on the frequency interactions suggested in the proof of the reduced bilinear estimates. By considering high-high-high to high interaction, we prove that the regularity requirement is sharp. Analogously by considering high-low-low to high interaction and low-low-low to low we obtain the upper and lower limit respectively for the sharp range of the parameter $\beta$.\\

As seen in Figure \ref{fig-1}, the result of Theorem \ref{Th-ill-Line} shows that the local well-posedness established in Section \ref{SUB:well} is sharp for $\alpha \geq \frac32$ or $\beta > \frac{1}{8} - \frac{\alpha}{4}$, as the contraction mapping argument employed in its proof yields a smooth data-to-solution map. However, for other ranges of $\beta$, it remains an open question whether the result obtained in this work is sharp. This sort of ill-posedness result was first introduced by Bourgain in \cite{B-97} in the KdV context and later used by several authors for other models; see for example \cite{LPS-18, MST-02, RV-17, SS-23} and references therein.
 


\subsection{Derivative nonlinear Schr\"odinger equations}
\label{SUB:dNLS}
 
Considering $v=D_x^\beta u$, we see that the IVP \eqref{MMTEquation} transforms to
\begin{equation} \label{MMT-v}
\begin{cases}
i\del_t v + (-\partial_x^2)^{\frac{\alpha}{2}}v = D_x^{2\beta} (|v|^2 v),\\
v(x, 0) = v_0(x),
\end{cases}\qquad x, t\in\R,
\end{equation}

\noi 
In particular, for $\alpha = 2$ and $\beta < \tfrac12$, the equation \eqref{MMT-v} was studied in \cite{WE-22}, where the authors proved that \eqref{MMT-v} is well-posed in $H^s(\R)$ for $s > \beta$, and ill-posed for $s < \beta$ in the sense of failure of uniform continuity of the solution map. However, the critical regularity $s = \beta$ was left unresolved.
Theorem~\ref{Th-loc-1} addresses this critical case and establishes well-posedness in $H^\beta$ (after the transformation $v = D_x^\beta u$). In particular, the local well-posedness result for the IVP \eqref{MMTEquation} in $H^s$, $s \ge 2\beta$, directly implies the local well-posedness of \eqref{MMT-v} in $H^s$, $s \ge \beta$.
In particular, we have the following consequence of Theorem \ref{Th-loc-1} and Remark \ref{RMK:high},
\begin{corollary}
\label{COR:dNLS}
Let $\al > 1$ and $0 \leq \beta < \frac{\al-1}2$.
Then, for initial data $v_0 \in H^s (\R)$, the IVP~\eqref{MMT-v} is locally well-posed in $H^s (\R)  $ for $s \ge \beta + \frac{2-\al}4$.
\end{corollary}

We remark that Corollary \ref{COR:dNLS} improves the well-posedness result in \cite{WE-22} by covering the endpoint case $s = \be$.

\medskip

When $\al = 2$ and $\beta = \tfrac12$, the equation \eqref{MMT-v} reduces to
\begin{align}\label{dNLS-n}
i\partial_t v - \partial_x^2 v = D_x\big(|v|^2 v\big),
\end{align}
which resembles the derivative Schr\"odinger equation
\begin{align}\label{dNLS}
i\partial_t v - \partial_x^2 v = i\big(|v|^2 v\big)_x.
\end{align}
It is known that \eqref{dNLS} is locally well-posed in $H^s(\R)$ for
$s \ge \tfrac12$, as shown in \cite{Tak99} via a gauge transform, and
ill-posed for $s < \tfrac12$ by \cite{BL01}. In contrast, implementing an
analogous gauge transform for \eqref{dNLS-n} is considerably more delicate,
due to the presence of the nonlocal derivative $D_x$. We plan to address this
problem in forthcoming work.

\medskip

More generally, for the fractional derivative NLS \eqref{MMT-v}, we have the
following consequence of Theorem~\ref{Th-ill-Line}.

\begin{corollary}
\label{COR:ill}
Let $\al > 1$. The initial value problem \eqref{MMT-v} is ill-posed in the
sense that the solution map fails to be $C^3$ if $\be < -\frac14$ or
$\be > \frac{\al-1}{2}$.
\end{corollary}

The endpoint $\be = - \frac14$ has appeared in \cite{MMT-97}.
As in our case, the endpoint $\be = \frac{\al -1}2$ has also been left open in \cite{WE-22} even for $\al = 2$.
It is expected that the ill-posedness range in Corollary~\ref{COR:ill} is
sharp.
More generally,
we consider 

\begin{conjecture}\label{CON:well}
Let $\al > 1$ and $\beta \in \big[-\frac14, \frac{\al-1}{2}\big]$. Then there
exists $s_0 \in \R$ such that \eqref{MMT-v} is well-posed in $H^s$ if and only if
$s \ge s_0$.
\end{conjecture}


More broadly, for a dispersive equation with a dispersion relation of order
$\al$, such as
\[
i\partial_t v + (-\partial_x^2)^{\frac{\al}{2}} v = 0
\qquad\text{or}\qquad
\partial_t v + \mathcal{H}(-\partial_x^2)^{\frac{\al}{2}} v = 0,
\]
where $\mathcal{H}$ denotes the Hilbert transform (the latter being the
classical dispersion generalised Benjamin--Ono equation \cite{CKS-02, ZG-12}), one may
view Conjecture~\ref{CON:well} as imposing a natural constraint on the
strength of derivative nonlinearities: the total number of derivatives falling
on the nonlinearity should not exceed $\al-1$.

Beyond \eqref{dNLS-n}, there are several other examples of endpoint-type
fractional dispersive PDEs of interest. One example is a third-order NLS with
a derivative cubic nonlinearity,
\[
i\partial_t v + D^3 v = \partial_x^2\big(v^3\big),
\]
while the cubic third-order NLS without derivatives has been studied
extensively; see \cite{CP-24, MT-18} and references therein. Other examples include the modified Benjamin--Ono
equation
\begin{align}\label{mBO}
\partial_t v - \mathcal{H}\partial_x^2 v = \partial_x\big(v^3\big),
\end{align}
the modified KdV equation
\[
\partial_t v - \partial_x^3 v = \partial_x^2\big(v^3\big),
\]
and, more generally, higher-order KdV-type equations of the form
\[
\partial_t v - \partial_x^{2k+1} v = \partial_x^{2k}\big(v^3\big),
\qquad k \ge 1 \ \text{an integer}.
\]
Resolving Conjecture~\ref{CON:well} would yield a satisfactory solution theory
for all of the above-mentioned equations.

In the subcritical-derivative regime $-\frac14<\be < \frac{\al-1}{2}$,
Theorems~\ref{Th-loc-1} and \ref{Th-ill-Line} verify Conjecture~\ref{CON:well}
whenever $2\be + \frac{2-\al}{4} \geq 0 $. 
However, in view of the failure of
uniform continuity for the solution maps of the derivative NLS \eqref{dNLS} and
the modified Benjamin–Ono equation \eqref{mBO}, one should not expect to prove
the conjecture in full via a simple contraction argument of the type used in this paper. Instead, techniques such as short-time Fourier restriction norm
methods \cite{MT-22}, or refined energy estimates \cite{MV-15}, may be better suited to handle these cases.
We will address this in our further work.

\begin{remark}\rm
Fractional nonlinear Schr\"odinger (fNLS) equation with derivative nonlinearity appears naturally in fractional quantum mechanics, see for example \cite{ Las-02, NL-2000}. Recently, Kato et al. \cite{KKO-25} considered the IVP for the fNLS equation
\begin{equation}
    \partial_t u +iD^{\alpha}u = F(u, \partial_xu, \overline{u}, \overline{\partial_xu}),
\end{equation}
with general nonlinearity that contains at most the first order derivative term. The authors in \cite{KKO-25} explored the structure of the nonlinearity and obtained a necessary and sufficient condition for the well-posedness of the IVP in the periodic case. It is worth  noticing that the MMT model \eqref{MMTEquation} possesses nonlinearity with a fixed structure and may involve fractional order derivatives. It would be interesting to have an analysis in the sprit of \cite{KKO-25} for the MMT model as well. We plan to pursue this direction in our ongoing project.

\end{remark}

\medskip
  
  \noindent
{\bf Organisation of the paper:} 
The remainder of the paper is organised as follows. 
Section~\ref{sec-2} introduces the function spaces and collects the auxiliary estimates needed for the proofs of our main results. 
In particular, we first state Proposition~\ref{prop-1} without proof, and then prove the well-posedness results—Theorem~\ref{Th-loc-1}, Corollary~\ref{Th-loc-2}, and Theorem~\ref{Th-loc-3}—using Proposition~\ref{prop-1}.
The proof of Proposition~\ref{prop-1} is divided into two sections. 
In Section~\ref{sec-3}, we reduce Proposition~\ref{prop-1} to bilinear estimates via duality and symmetry; establishing these bilinear estimates yields Proposition~\ref{prop-1} for $s \ge 0$. 
Section~\ref{sec-4} is devoted to a trilinear estimate for certain frequency interactions, which is then used to prove Proposition~\ref{prop-1} for $s<0$. 
Finally, the ill-posedness results are proved in Section~\ref{sec-5}. \\

\section{Function spaces and preliminary estimates}\label{sec-2}

In this section, we introduce some notation, function spaces and basic results that will be used throughout the text.  For $s\in \R$, the classical Sobolev spaces are defined by
$$ H^s(\R) = \{f\in \mathcal{S}': \|f\|_{H^s(\R)}:= \|\langle \xi\rangle^s \widehat{f}\|_{L^2(\R)}<\infty\},$$
where $\langle\cdot\rangle = (1+|\cdot|^2)^{\frac12}$  and $\widehat{f}$ is the usual Fourier transform of the function $f$. We define the  group associated to the linear problem by
$$S(t)u_0 = (e^{it|\xi|^{\alpha}}\widehat{u_0})^{\vee}.$$
Also, for $s, b\in \R$, we define the Fourier restriction norm space 
\begin{equation*}
X^{s,b}(\R) = \{f\in \mathcal{S}': \|f\|_{X^{s,b}}:= \|\langle\xi\rangle^s\langle\tau-|\xi|^{\alpha}\rangle^{b}\widehat{f}(\xi, \tau)\|_{L^2(\R^2)}<\infty\}.
\end{equation*}

In what follows, we record some linear and nonlinear estimates satisfied by the solution in  $X^{s,b}$  spaces. First we define a cut-off function $\psi_1 \in C^{\infty}(\R;\; \R^+)$ which is even such that $0\leq \psi_1\leq 1$ and
\begin{equation*}\label{cut-1}
\psi_1(t) = \begin{cases} 1, \quad |t|\leq 1,\\
                          0, \quad |t|\geq 2.
            \end{cases}
\end{equation*}
We also define $\psi_T(t) = \psi_1(t/T)$, for $0< T\leq 1$.

We now move to state some standard results that are crucial in the proof of the local well-posedness result.
\begin{lemma}\label{lemma1}
For any $s, b \in \R$, we have
\begin{equation*}
\|\psi_1S(t)\phi\|_{X^{s,b}}\leq C \|\phi\|_{H^s}.
\end{equation*}
Further, if  $-\frac12<b'\leq 0\leq b<b'+1$ and $0\leq T\leq 1$, then
\begin{equation*}
\|\psi_{T}\int_0^tS(t-t')f(u(t'))dt'\|_{X^{s,b}}\lesssim T^{1-b+b'}\|f(u)\|_{X^{s, b'}}.
\end{equation*}
\end{lemma}
\begin{proof}
The proof of this lemma is now standard; see, for instance, \cite{GTV}.
\end{proof}

\begin{remark} In the proof of the local well-posedness results, we will take $b'=-\frac12+2\epsilon$ and $b=\frac12+\epsilon$ so that $1-b+b'$ is strictly positive.
\end{remark}

Now, we state the trilinear estimate that is central in our argument.

\begin{proposition}\label{prop-1}
Let $1<\alpha\leq 2$, $-\frac{1}{4} < \beta < \frac{\alpha-1}{2}$, b and $b'$ be as in Lemma \ref{lemma1}. Then  the following trilinear estimate
\begin{equation}\label{tlinear}
\|D_x^\beta(|D_x^\beta u|^2 D_x^\beta u )\|_{X^{s,b'}(\R)}\lesssim  \|u\|_{X^{s,b}(\R)}^3,
\end{equation}
holds in the following cases:
\begin{itemize}
\item $s \geq 2\beta + \frac{2-\alpha}{4}$ for $\frac18 - \frac{\alpha}{4} < \beta < \frac{\alpha -1}{2}$
\item $s > \beta + \frac{5}{8} - \frac{\alpha}{2}$ for $\frac{7}{8} - \alpha < \beta \leq \frac{1}{8} - \frac{\alpha}{4}$
\item $ s \geq \frac{4}{3}\beta + \frac{2-\alpha}{6}$ for $-\frac14 < \beta \leq \frac{7}{8} - \alpha$.
\end{itemize}
\end{proposition}

The proof of this proposition will be provided below in a separate section; Section \ref{sec-3} for $s \ge 0$ cases and Section \ref{sec-4} for $s < 0$ cases. 
For the sake of completeness, we use Proposition~\ref{prop-1} to prove the local well-posedness results stated above. 

\begin{proof}[Proof of Theorem \ref{Th-loc-1}, Corollary \ref{Th-loc-2}, and Theorem \ref{Th-loc-3}] Let $1<\alpha\leq2$, $s_{\beta, \alpha}$ as defined in \eqref{s-beta-line},  $s\geq s_{\beta,\alpha}$ and $u_0\in H^s$. Using the linear estimates from Lemma~\ref{lemma1} and trilinear estimate from Proposition~\ref{prop-1} with $b=\frac12+\varepsilon$ and $b'=-\frac12+2\varepsilon$, by standard argument we can prove that the application
$$\Phi_t(u) = S(t) u_0-\int_0^t S(t-t') D_x^\beta(|D_x^\beta u|^2 D_x^\beta u) dt'$$
is a contraction map in an appropriate ball of $H^s\cap X^{s,b}_T$ for appropriate $0<T\leq 1$. So, we omit the details.
\end{proof}

\section{Bilinear estimates}
\label{sec-3}

In this section, 
we prove Proposition \ref{prop-1} with non-negative regularity $s \ge 0$ by reducing it to a bilinear estimate.
Before providing a proof for Proposition \ref{prop-1}, we will prove a bilinear estimate that plays a fundamental role. To achieve this, we will follow the multipliers techniques introduced by Tao in \cite{TT-01}. For the sake of completeness, we start by recording some notation and settings from \cite{TT-01}.  For any $k\geq 2$ and any abelian additive group $G:= \R\times\R$ with invariant measure $d\xi$, we denote by $\Gamma_k(G)$ the hyperplane defined by
$$\Gamma_k(G):= \{\xi=(\xi_1, \xi_2,\cdots, \xi_k)\in G^k :  \xi_1+\xi_2+\cdots + \xi_k=0\},$$
endowed with the  measure
$$\int_{\Gamma_k(G)}f:= \int_{G^{k-1}}f(\xi_1, \xi_2, \cdots, \xi_{k-1}, -\xi_1-\xi_2-\cdots -\xi_{k-1}) d\xi_1\cdots d\xi_{k-1}.$$
We call any function $m: \Gamma_k(G):\to\C$ a $[k; G]$-multiplier. Given such a multiplier $m$, we define its norm $\|m\|_{[k;G]}$ as the best constant such that, for all test functions $f_j\in G$,  the following inequality holds
$$\Big|\int_{\Gamma_k(G)}  m(\xi)\prod_{j=1}^kf_j(\xi_j)\Big|\leq \|m\|_{[k;G]}\prod_{j=1}^k \|f_j\|_{L^2(G)}.$$

For some basic properties of this operator norm $\|m\|_{[k;G]}$, we refer readers to \cite[Section 3]{TT-01}. Now we record some more notations and summation conventions that will be used throughout this text. We will use $N_j$, $L_j$  and $H$ to denote dyadic variables. 
For $N_1, N_2, N_3 >0$, we define $N_{\max}\geq N_{\med}\geq N_{\min}$ as the maximum, medium and minimum of  $N_1, N_2, N_3$. Similarly, we define $L_{\max}\geq L_{\med}\geq L_{\min}$ in the case when  $L_1, L_2, L_3>0$.

While proving a bilinear estimate in the $X^{s,b}$ spaces, as will be clear later, we need to find $\|\widetilde{m}\|_{[3; G]}$-norm of the following multiplier
$$\widetilde{m}:= \frac{m(\xi_1, \xi_2, \xi_3)}{\prod_{j=1}^3\langle\lambda_j\rangle^{b_j}},$$
where $\lambda_j:=\tau_j- h_j(\xi_j)$ and $h_j(\xi_j)$ is the dispersion relation under consideration.

 As described in \cite{TT-01}, using the comparison principle and the averaging argument to estimate the $\|\widetilde{m}\|_{[3; G]}$-norm, one can suppose $|\lambda_j|\ges 1$ and $\max(|\xi_1|, |\xi_2|, |\xi_3|)\ges 1$.  Performing the dyadic decomposition of the variables $\xi_j$, $\lambda_j$ and 
 \begin{align} \label{H}
 h(\xi) = h_1(\xi_1)+ h_2(\xi_2)+h_3(\xi_3),
 \end{align}
 one has that
 \begin{equation*}
 \|\widetilde{m}\|_{[3; G]}\les \Big\| \sum_{N_{\max}\ges 1}\sum_{H}\sum_{L_1, L_2, L_3\ges 1}\frac{\mathbf{m}(N_1,N_2, N_3)}{L_1^{b_1}L_2^{b_2}L_3^{b_3}}\chi_{N_1,N_2,N_3,H,L_1, L_2,L_3}\Big\|_{[3; G]},
 \end{equation*}
 where $\chi_{N_1,N_2,N_3,H,L_1, L_2,L_3}$ is the multiplier
 \begin{equation}\label{char-1}
 \chi_{N_1,N_2,N_3,H,L_1, L_2,L_3}(\xi, \tau) := \chi_{|h(\xi)|\sim H}\prod_{j=1}^3\chi_{|\xi_j|\sim N_j} \chi_{|\lambda_j|\sim L_j}
 \end{equation}
 and 
 $$\mathbf{m}(N_1,N_2, N_3):= \sup_{|\xi_j|\sim N_j, \forall\, j=1,2,3}|m(\xi_1, \xi_2, \xi_3)|.$$

Note that,  on the support of the multipliers, we have $\xi_1+\xi_2+\xi_3=0$, $\tau_1+\tau_2+\tau_3=0$ and the resonance  relation $h(\xi) +\lambda_1 + \lambda_2 + \lambda_3=0$. In view of these relations, the multiplier $ \chi_{N_1,N_2,N_3,H,L_1, L_2,L_3}(\xi, \tau) $ vanishes unless 
\begin{equation*}
\Nmax \sim \Nmed \hspace{1cm}\text{and} \hspace{1cm} \Lmax \sim \max(H, \Lmed).
\end{equation*}
As described above, considering  $N_1 \geq N_2 \geq N_3$,  we have $N_1 \sim N_2 \gtrsim 1$. Therefore, using triangle inequality and Schur's test,  to compute $\|\widetilde{m}\|_{[3; G]}$, one needs to estimate
\begin{equation*}
 \sup_{N\gtrsim 1}  \sum_{\Nmax \sim \Nmed \sim N} \sum_H \sum_{\Lmax \sim \max(H, \Lmed)} \frac{\mathbf{m}(N_1,N_2, N_3)}{L_1^{b_1}L_2^{b_2}L_3^{b_3}} \big\|  \chi_{N_1,N_2,N_3;H;L_1,L_2,L_3} \big\|_{[3; G]}.
 \end{equation*}

In this way, estimates for the multiplier $ \chi_{N_1,N_2,N_3,H,L_1, L_2,L_3}(\xi, \tau) $ play a crucial role in this argument. In the following proposition, we record the estimates from \cite{CHKL-15} for the dispersion relation $h_j(\xi_j) = \pm |\xi_j|^\alpha$, $1<\alpha \leq 2$. Before stating the result, let us first note that from the resonance relation $h(\xi_1, \xi_2, \xi_3) = |\xi_1|^\alpha - |\xi_2|^\alpha + |\xi_3|^\alpha$, it is true that $\Nmax^{\alpha-1}\Nmin \lesssim H \lesssim \Nmax^\alpha$. 

\begin{prop}[Bilinear multiplier Estimates] \label{CountingEst}
Let $H, N_1,N_2,N_3, L_1, L_2, L_3$ be dyadic numbers and  for $1<\alpha \leq2 $, $h(\xi) = |\xi_1|^\alpha - |\xi_2|^\alpha + |\xi_3|^\alpha$ be the resonance relation. Then 
we have the following
\begin{itemize}
\item If $H \sim \Lmax \sim L_1$ and $N_1 \sim \Nmax$,
\[\big\| \chi_{N_1,N_2,N_3;H;L_1,L_2,L_3} \big\|_{[3; \R\times\R]}\lesssim \Lmin^\frac{1}{2} \min(\Nmin^\frac{1}{2}, \Nmax^{\frac{1-\alpha}{2}}\Lmed^\frac{1}{2}).\]

\item If $H \sim \Lmax \sim L_1$ and $N_2 \sim N_3 \gg N_1$, 
\[\big\| \chi_{N_1,N_2,N_3;H;L_1,L_2,L_3} \big\|_{[3; \R\times\R]} \lesssim \Lmin^\frac{1}{2} \min(\Nmin^\frac{1}{2}, \Nmax^{\frac{2-\alpha}{2}}\Nmin^{-\frac{1}{2}}\Lmed^\frac{1}{2}).\]

\item If $H \sim \Lmax \sim L_2$ and $\Nmax \sim \Nmin$, 
\[\big\| \chi_{N_1,N_2,N_3;H;L_1,L_2,L_3} \big\|_{[3; \R\times\R]} \lesssim \Lmin^\frac{1}{2}\Nmin^\frac{1}{2}.\]

\item If $H \sim \Lmax \sim L_2$ and $\Nmax \sim \Nmed \gg \Nmin$, 
\[ \big\| \chi_{N_1,N_2,N_3;H;L_1,L_2,L_3} \big\|_{[3; \R\times\R]} \lesssim \Lmin^\frac{1}{2} \min(\Nmin^\frac{1}{2}, \Nmax^{\frac{1-\alpha}{2}} \Lmed^\frac{1}{2}). \]

\item If $H \ll \Lmax \sim \Lmed$, 
\[ \big\| \chi_{N_1,N_2,N_3;H;L_1,L_2,L_3} \big\|_{[3; \R\times\R]} \lesssim \Lmin^\frac{1}{2}  \Nmin^\frac{1}{2}.\]
\end{itemize}
By symmetry, the same estimates hold for the case $H \sim \Lmax \sim L_3$
\end{prop}

We now move to prove the bilinear estimate that will lead to the proof of Proposition~\ref{prop-1} for non-negative regularity, the key ingredient for the proof of Theorem~\ref{Th-loc-1} and Corollary~\ref{Th-loc-2}, as announced earlier. 
\begin{prop}\label{bilinear} Let $1<\alpha\leq2$, $-\frac{1}{4} \le \beta < \frac{\alpha-1}{2}$, $\widehat{J_x^{\beta}u}(\xi)=\langle\xi\rangle^{\beta}\widehat{u}(\xi)$ and $\varepsilon>0$ arbitrarily small. Then for all $s\geq 2\beta + \frac{2-\alpha}{4}$ the following bilinear estimate holds
\begin{equation}\label{bil-1}
\|J_x^{\beta}u J_x^{\beta}v\|_{L^2(\R)} =\|J_x^{\beta}u \overline{J_x^ {\beta}v}\|_{L^2(\R)} \leq \|u\|_{X^{0, \frac12-2\epsilon}(\R)}\|v\|_{X^{s,\frac12+\epsilon}(\R)}.
\end{equation}
\end{prop}

\begin{remark}
\rm 
When $\beta = 0$, Proposition \ref{bilinear} was previously established in \cite[Proposition 3.2]{CHKL-15}. In this work, we extend the result to a broader range of $\beta$. In particular, when $\beta > 0$, estimate \eqref{bil-1} exhibits a `smoothing' effect, which enables us to treat derivative nonlinearities in \eqref{MMTEquation}. For fixed $\alpha$, the upper bound $\beta < \frac{\alpha - 1}{2}$ is expected to be sharp; see Theorem \ref{Th-ill-Line} for a counterexample in the case $\beta > \frac{\alpha - 1}{2}$.
\end{remark}

\begin{proof}

Using the duality and standard properties of the Fourier transform, the estimate \eqref{bil-1} reduces to bounding the following multiplier expression:
\begin{equation} \label{multiplier}
\bigg\| \frac{\langle \xi_1\rangle^\beta \langle\xi_2\rangle^\beta}{\abrac{\xi_1}^s \abrac{\tau_1 - |\xi_1|^\alpha}^{\frac{1}{2}+\varepsilon} \abrac{\tau_2 + |\xi_2|^\alpha}^{\frac{1}{2}-2\varepsilon}} \bigg\|_{[3; \R\times\R ]}.
\end{equation}

For the dyadic numbers $N_j$, $L_j$ and $H$, recall the multiplier $\chi_{N_1,N_2,N_3;H;L_1,L_2,L_3}$ given in \eqref{char-1} with $h_j(\xi_j) = \pm |\xi_j|^\alpha$ and $\lambda_j := \tau_j - h_j(\xi_j)$.
Then, for estimating \eqref{multiplier}, it is instead sufficient to estimate 
\begin{equation} \label{CharNormSum}
 \Bigg\| \sum_{\Nmax \geq \Nmed \geq \Nmin} \sum_{\Lmax \geq \Lmed \geq \Lmin \gtrsim 1} \sum_{H} \frac{\abrac{N_1}^\beta \abrac{N_2}^\beta}{\abrac{N_1}^s L_1^{\frac{1}{2}+\varepsilon} L_2^{\frac{1}{2}-2\varepsilon}}  \chi_{N_1,N_2,N_3;H;L_1,L_2,L_3} \Bigg\|_{[3; \R\times\R]}. 
\end{equation}
From the identities $\xi_1 + \xi_2 + \xi_3 = 0 $ and $\lambda_1 + \lambda_2 + \lambda_3 + h(\xi_1, \xi_2, \xi_3) = 0$, we see the multiplier \eqref{char-1} vanishes unless 
\begin{equation}\label{cond-1}
\Nmax \sim \Nmed \hspace{1cm}\text{and} \hspace{1cm} \Lmax \sim \max(H, \Lmed).
\end{equation}


\noi 
For the resonance relation \eqref{H}, 
it follows that 
\begin{align} \label{H2}
\Nmax^{\alpha-1}\Nmin \lesssim H \lesssim \Nmax^\alpha.
\end{align}
See Lemma \ref{Reso-Lemma} for a justification of \eqref{H2}.
Using \eqref{cond-1} and Schur's test (Lemma 3.11 in \cite{TT-01}), estimating \eqref{CharNormSum}   reduces  to bounding
\begin{equation}\label{est-m1}
\sup_{N\gtrsim 1}  \sum_{\Nmax \sim \Nmed \sim N} \sum_H \sum_{\Lmax \sim \max(H, \Lmed)} \frac{\abrac{N_1}^\beta \abrac{N_2}^\beta}{\abrac{N_1}^s L_1^{\frac{1}{2}+\varepsilon} L_2^{\frac{1}{2}-2\varepsilon}} \big\|  \chi_{N_1,N_2,N_3;H;L_1,L_2,L_3} \big\|_{[3; \R\times\R]}.
\end{equation}

In what follows, in the light of \eqref{cond-1}, we bound \eqref{est-m1} by estimating the following expressions arising, respectively, from high and low modulation parts
\begin{equation} \label{Hll}
\sum_{\Nmax \sim \Nmed \sim N} \sum_{\Lmax \sim \Lmed \gtrsim 1} \sum_{H \ll \Lmax} \frac{\abrac{N_1}^\beta \abrac{N_2}^\beta}{\abrac{N_1}^s L_1^{\frac{1}{2}+\varepsilon} L_2^{\frac{1}{2}- 2\varepsilon}} \big\|  \chi_{N_1,N_2,N_3;H;L_1,L_2,L_3} \big\|_{[3; \R\times\R]}
\end{equation}
and
\begin{equation}\label{Hsim}
\sum_{\Nmax \sim \Nmed \sim N} \sum_{L_1,L_2,L_3 \gtrsim 1} \sum_{H \sim \Lmax} \frac{\abrac{N_1}^\beta \abrac{N_2}^\beta}{\abrac{N_1}^s L_1^{\frac{1}{2}+\varepsilon} L_2^{\frac{1}{2}- 2\varepsilon}}\big\|  \chi_{N_1,N_2,N_3;H;L_1,L_2,L_3} \big\|_{[3; \R\times\R]}.
\end{equation}

\noi 
{\bf Case 1}.
We start by dealing with the high modulation term \eqref{Hll}. By Proposition \ref{CountingEst}, using that $L_1^{\frac{1}{2}+\varepsilon} L_2^{\frac{1}{2}-2\varepsilon} \gtrsim \Lmin^{\frac{1}{2}+\varepsilon}\Lmed^{\frac{1}{2}-2\varepsilon}$ and taking in consideration $\Lmax \gtrsim \Nmin\Nmax^{\alpha-1}$, we have
\begin{equation}\label{Sum2.2}
\begin{split}
\eqref{Hll} &\lesssim \sum_{\Nmax \sim \Nmed \sim N} \sum_{\Lmax \sim \Lmed \gtrsim N^{\alpha-1}\Nmin } \frac{\abrac{N_1}^\beta \abrac{N_2}^\beta}{\abrac{N_1}^s \Lmin^{\frac{1}{2} + \varepsilon} \Lmed^{\frac{1}{2}-2\varepsilon}}\Lmin^\frac{1}{2}\Nmin^\frac{1}{2} \log \Lmax  \\
& \lesssim \sum_{\Nmax \sim \Nmed \sim N}\sum_{\Lmin\geq 1} \frac{\abrac{N_1}^\beta \abrac{N_2}^\beta}{\abrac{N_1}^s \Lmin^{\varepsilon} \Nmin^{\frac12 - 2\varepsilon}N^{\frac{\alpha-1}{2}-2(\alpha-1)\varepsilon}}\Nmin^\frac{1}{2}N^\varepsilon \\
& \lesssim \sum_{\Nmax \sim \Nmed \sim N}  \frac{\abrac{N_1}^\beta \abrac{N_2}^\beta}{\abrac{N_1}^s N^{\frac{\alpha-1}{2}-2(\alpha-1)\varepsilon}}\Nmin^{2\varepsilon}N^\varepsilon.
\end{split}
\end{equation}

We bound \eqref{Sum2.2}, considering two different cases.\\

\noindent
{\it Case 1.1}.  $N_1 = \Nmin$ or $N_2 = \Nmin$: First consider $N_1 = \Nmin$. With this consideration, we can bound \eqref{Sum2.2} by
\begin{equation}\label{Sum1}
 \eqref{Sum2.2} \les \sum_{\Nmin}  \abrac{\Nmin}^{\beta-s } N^{\beta+\frac{1-\alpha}{2}+(2\alpha+1)\varepsilon}  .
 \end{equation}

If $\Nmin\les 1$, then the sum in \eqref{Sum1} is bounded by $N^{\beta+\frac{1-\alpha}{2}+(2\alpha+1)\varepsilon}$ requiring the condition $\beta<\frac{\alpha-1}{2}$ to obtain the required uniform bound.

If  $\Nmin \gg 1$,  then one can estimate \eqref{Sum1} as
\begin{equation*}
   \eqref{Sum1}\les \begin{cases} N^{2\beta +\frac{1-\alpha}{2} - s + (2\alpha + 1)\varepsilon}, \quad &{\textrm{if}}\; \beta-s\geq 0,\\
   N^{\beta +\frac{1-\alpha}{2}+(2\alpha+1)\varepsilon}, &{\textrm{if}}\; \beta-s< 0,
   \end{cases}
\end{equation*}
which is bounded
requiring condition $s > 2\beta + \frac{1-\alpha}{2}$ and $\beta<\frac{\alpha-1}{2}$ to obtain the desired uniform bound for \eqref{Hll}.

Now consider $N_2 = \Nmin$. In this case, one can control \eqref{Sum2.2} as
\begin{equation}\label{Sum11}
 \eqref{Sum2.2} \les \sum_{\Nmin} \abrac{\Nmin}^{\beta} N^{\beta+\frac{1-\alpha}{2}-s+(2\alpha+1)\varepsilon}
 \les
 \begin{cases}
 N^{2\beta +\frac{1-\alpha}{2}- s + (2\alpha + 1)\varepsilon}, \quad &{\textrm{if}}\; \beta \geq 0 \\
 N^{\beta + \frac{1-\alpha}{2} -s +(2\alpha+1)\varepsilon}, \quad &{\textrm{if}}\; \beta < 0,
 \end{cases}
\end{equation}
leaving conditions $s > 2\beta + \frac{1 -\alpha}{2}$ for $\be \ge 0$, and $s > \be + \frac{1-\al}2$ for $\be < 0$, to get a uniform bound for \eqref{Hsim}. 
In both cases, we have $s \ge 2 \be + \frac{2-\al}4$ for $\beta \in [-\frac14, \frac{\al -1}2)$ provided $\al > 1$, which is sufficient for our purpose.\\

\noindent
{\it Case 1.2}.  $N_1 \sim N_2 \sim \Nmax$: In this case, we have 
\begin{equation} 
 \eqref{Sum2.2} \les N^{2\beta +\frac{1-\alpha}{2} - s + (2\alpha + 1)\varepsilon}
\end{equation} 
requiring only the conditions $s > 2\beta + \frac{1-\alpha}{2}$ to obtain the required uniform bound for \eqref{Hll}.\\

\noi 
{\bf Case 2}.
We now move on to estimate the low modulation term \eqref{Hsim}.\\

\noi 
{\it Case 2.1}. 
We first consider the cases $L_1 = \Lmax$ and $N_1 = \Nmax$. By Proposition \ref{CountingEst}, 
one has
\begin{equation}\label{L1MaxN1Max}
\begin{split}
\eqref{Hsim} 
& \lesssim \sum_{\Nmax \sim \Nmed \sim N} \sum_{L_1, L_2, L_3 \gtrsim 1}\frac{\abrac{N}^\beta \abrac{N_2}^\beta}{\abrac{N}^s \Lmax^{\frac{1}{2}+\varepsilon} L_2^{\frac{1}{2}-2\varepsilon}}\Lmin^\frac{1}{2} \min( \Nmin^\frac{1}{2}, \Nmax^{\frac{1-\alpha}{2}} \Lmed^\frac{1}{2}) \\
& \lesssim \sum_{\Nmax \sim \Nmed \sim N} \sum_{\Lmax \geq \Lmed \gtrsim 1}\frac{\abrac{N}^\beta \abrac{N_2}^\beta}{\abrac{N}^s \Lmax^{\frac{1}{2}+\varepsilon}}\Lmed^{2\varepsilon} \min( \Nmin^\frac{1}{2}, \Nmax^{\frac{1-\alpha}{2}} \Lmed^\frac{1}{2})  \\
& \lesssim \sum_{\Nmax \sim \Nmed \sim N} \sum_{ \Lmed \gtrsim 1}\frac{\abrac{N}^\beta \abrac{N_2}^\beta}{\abrac{N}^s} \Lmed^{\varepsilon} \min( \Nmin^\frac{1}{2}\Lmed^{-\frac{1}{2}}, \Nmax^{\frac{1-\alpha}{2}})   
\end{split}
\end{equation}
summing over $\Lmin$ the $\Lmax$ in the last two lines. \\

To estimate \eqref{L1MaxN1Max}, we  split the $\Nmin$ sum over two sets $\{ 0 < \Nmin < N^{1-\alpha}\}$ and $\{ N^{1-\alpha} \leq \Nmin \lesssim N\}$.
In the first set, the sum in \eqref{L1MaxN1Max} is bounded by
\begin{equation}\label{Sum3.3}
    \sum_{0< \Nmin < N^{1-\alpha}}\sum_{\Lmed  \ges 1}\frac{\abrac{N}^\beta}{\abrac{N}^s}\abrac{N_2}^\beta \Nmin^\frac{1}{2} L_{\med}^{ - \frac12 + \eps} . 
 \end{equation}
 
Then summing over $\Lmed  $ and $\Nmin$, the sum in \eqref{Sum3.3} can be estimated as
\begin{equation*} 
\eqref{Sum3.3}\les \begin{cases} N^{2\beta + \frac{1-\alpha}2 - s  }, \qquad & {\textrm{if}}\; \beta\geq 0,\\
N^{\beta +  \frac{1-\alpha}2 - s}, & {\textrm{if}}\; \beta < 0,
\end{cases}
\end{equation*}
which, similarly to \eqref{Sum11}, is bounded.

For the sum over the second set $\{ N^{1-\alpha} \leq \Nmin \lesssim N\}$, we need to split up the $\Lmed$ sum as follows
\begin{equation*}
\begin{split}
\eqref{L1MaxN1Max} &\lesssim \sum_{N^{1-\alpha} \leq \Nmin \lesssim N}\bigg( \sum_{1 \leq \Lmed < \Nmin N^{\alpha-1}} \frac{\abrac{N}^\beta}{\abrac{N}^s} \abrac{N_2}^\beta N^{\frac{1-\alpha}{2}}\Lmed^{\varepsilon} \\
& \qquad+\sum_{\Lmed \geq \Nmin N^{\alpha-1}} \frac{\abrac{N}^\beta}{\abrac{N}^s}\abrac{N_2}^\beta \Nmin^\frac{1}{2}\Lmed^{-\frac{1}{2}+\varepsilon} \bigg)\\
& \lesssim \sum_{N^{1-\alpha} \leq \Nmin \lesssim N} \frac{\abrac{N}^{\beta+\frac{1-\alpha}{2}}}{\abrac{N}^s}\abrac{N_2}^\beta (\Nmin N^{\alpha-1})^{\varepsilon} \\
&\les 
\begin{cases}
N^{2\beta + \frac{1-\alpha}{2} - s+\alpha\varepsilon}, \qquad & {\textrm{if}}\; \beta\geq 0,\\
N^{\beta +\frac{1-\alpha}{2} - s +\alpha\varepsilon}, & {\textrm{if}}\; \beta < 0,
\end{cases}
\end{split}
\end{equation*}
which is again sufficient for our purpose.\\

\noi 
{\it Case 2.2}.
Now consider the case $L_3 = \Lmax$ and $N_3 = \Nmax$. Since $N^{\alpha-1}\Nmin \sim H \sim L_3 \gtrsim 1$, we can reduce to $\Nmin \gtrsim N^{1-\alpha}$.\\

\noi 
{\it Case 2.2.1}.
If $N_1 = \Nmin$, then
\begin{equation*}
\begin{split} 
\eqref{Hsim} 
& \lesssim \sum_{N^{1-\alpha} \lesssim \Nmin \lesssim N} \sum_{\Lmax \geq \Lmed \geq \Lmin} \frac{\abrac{N}^\beta \abrac{\Nmin}^\beta}{\abrac{\Nmin}^s \Lmin^{\frac{1}{2}+\varepsilon} \Lmed^{\frac{1}{2}-2\varepsilon}}\Lmin^\frac{1}{2}N^{\frac{1-\alpha}{2}}\Lmed^\frac{1}{2} \\
& \lesssim \sum_{N^{1-\alpha} \lesssim \Nmin \lesssim N} \sum_{\Lmed \lesssim \Lmax \lesssim N^\alpha} \frac{\abrac{\Nmin}^\beta}{\abrac{\Nmin}^s}\abrac{N}^\beta \Lmed^{2\varepsilon}N^{\frac{1-\alpha}{2}}\\
&\lesssim
 \sum_{N^{1-\alpha} \lesssim \Nmin \lesssim N} \abrac{\Nmin}^{\beta-s} \abrac{N}^{\beta}N^{\frac{1-\alpha}{2}} N^{2(\alpha-1)\varepsilon}\Nmin^{2\varepsilon}\\
 &\les \begin{cases} N^{\beta + \frac{1-\alpha}{2} +2\alpha\varepsilon} + N^{2\beta +\frac{1-\alpha}{2} - s + 2\alpha\varepsilon}, \quad &{\textrm{if}}\; \beta-s\geq 0,\\
   N^{\beta +\frac{1-\alpha}{2}+2\alpha\varepsilon} &{\textrm{if}}\; \beta-s< 0,
   \end{cases}
\end{split}
\end{equation*}
by considering the sets $\{N^{1-\alpha} \lesssim \Nmin \leq 1\}, \{ 1 \leq \Nmin \lesssim N\}$ separately.
Imposing condition $s > 2\beta + \frac{1-\alpha}{2}$ and $\beta<\frac{\alpha-1}{2}$ to obtain a uniform bound for \eqref{Hsim}.
Therefore, if $\be \ge 0$, there is no restriction on the range of $s$; while for $\be < 0$, we have $s > 2 \be + \frac{1-\al}2 $, which is sufficient again.
\\

\noi 
{\it Case 2.2.2}.
Similarly, if $N_2=\Nmin$, then
\begin{equation*}
\begin{split} 
\eqref{Hsim} &\lesssim
 \sum_{N^{1-\alpha} \lesssim \Nmin \lesssim N} \abrac{\Nmin}^{\beta} \abrac{N}^{\beta-s}N^{\frac{1-\alpha}{2}} N^{2(\alpha-1)\varepsilon} \Nmin^{2\varepsilon}\\
 &\les \begin{cases}  N^{2\beta +\frac{1-\alpha}{2} - s + (2\alpha + 1)\varepsilon}, \quad &{\textrm{if}}\; \beta\geq 0,\\
   N^{\beta + \frac{1-\alpha}{2}-s+ (2\alpha + 1)\varepsilon} &{\textrm{if}}\; \beta< 0,
   \end{cases}
\end{split}
\end{equation*}
yielding a required uniform bound.\\

\noi 
{\it Case 2.3}.
Next, we deal with the case $L_2 = \Lmax$ and $N_2 \sim \Nmax$. In this case, by Proposition~\ref{CountingEst} and using $\Lmax \gtrsim \Nmin\Nmax^{\alpha-1}$, one has
\begin{equation}\label{Sum7}
\begin{split}
\eqref{Hsim} & \lesssim \sum_{\Nmin \lesssim \Nmax \sim N} \sum_{\Lmax \geq \Lmed \geq \Lmin \gtrsim 1}\frac{\abrac{N_1}^{\beta}\abrac{N}^\beta}{\abrac{N_1}^sL_1^{\frac{1}{2}+\varepsilon}\Lmax^{\frac{1}{2}-2\varepsilon}}\Lmin^\frac{1}{2}\Nmin^\frac{1}{2} \\
& \lesssim \sum_{\Nmin \lesssim \Nmax \sim N} \sum_{\Lmax \geq  \Lmin \gtrsim 1}\frac{\abrac{N_1}^{\beta}\abrac{N}^\beta}{\abrac{N_1}^sL_1^{\frac{1}{2}+\varepsilon}\Lmax^{\frac{1}{2}-3\varepsilon}}\Lmin^\frac{1}{2}\Nmin^\frac{1}{2} \\
&\lesssim  \sum_{\Nmin \lesssim \Nmax \sim N}\frac{\abrac{N}^{\beta} \abrac{N_1}^{\beta-s}}{N^{\frac{\alpha-1}{2} - 3(\alpha-1)\varepsilon}\Nmin^{\frac12 - 3\varepsilon}}\Nmin^\frac{1}{2} \\
& \lesssim N^{\beta + \frac{1-\alpha}{2} + 3(\alpha-1)\varepsilon}\sum_{\Nmin \lesssim N} \abrac{N_1}^{\beta-s}\Nmin^{3\varepsilon}. 
\end{split}
\end{equation}

\noi 
{\it Case 2.3.1}.
If $N_1 \sim \Nmax \sim N$, then we can bound \eqref{Sum7} by $N^{2\beta + \frac{1-\alpha}{2} - s + 2\alpha\varepsilon}$
imposing the condition $s > 2\beta + \frac{1-\alpha}{2}$.\\ 

\noi 
{\it Case 2.3.2}.
Otherwise, $N_1 \sim \Nmin$.
We first consider the case where $\Nmin \les 1$.
Then, we have 
$\eqref{Sum7} \lesssim  N^{\beta + \frac{1-\alpha}{2} + 2(\alpha-1)\varepsilon}$. 
When $\Nmin \geq 1$, we have
\begin{equation}
\eqref{Sum7} \lesssim  
\begin{cases}
N^{2\beta + \frac{1-\alpha}{2} - s + 2\alpha\varepsilon} &  \text{ if } \beta - s+ 3\varepsilon \geq 0 \\
N^{\beta + \frac{1-\alpha}{2} + 2(\alpha-1)\varepsilon} & \text{ if } \beta - s + 3\varepsilon < 0.
\end{cases}
\end{equation}
Both cases are sufficient for our purpose.
\\



\noi
{\it Case 2.4}.
Now, we move on to the case $L_1 =\Lmax$ and $N_2 \sim N_3 \gg N_1$. Since $1 \lesssim \Lmax \sim H \sim \Nmin N^{\alpha-1}$ we can reduce to $\Nmin \gtrsim N^{1-\alpha}$. Once again, using Proposition \ref{CountingEst}, we have
\begin{equation*}
\begin{split}
\eqref{Hsim} & \lesssim \sum_{\Nmax \sim \Nmed \sim N} \sum_{\Lmax \geq \Lmed \geq \Lmin \gtrsim 1} \frac{\abrac{\Nmin}^\beta \abrac{N}^\beta}{\abrac{\Nmin}^s L_1^{\frac{1}{2}+\varepsilon} L_2^{\frac{1}{2}-2\varepsilon}}\Lmin^\frac{1}{2}\Nmin^\frac{1}{2} \\
& \lesssim \sum_{\substack{\Nmax \sim \Nmed \sim N \\ \Nmin \gtrsim N^{1-\alpha}}} \frac{\abrac{\Nmin}^\beta \abrac{N}^\beta}{\abrac{\Nmin}^s (N^{\alpha-1}\Nmin)^{\frac{1}{2}+\varepsilon}} \Nmin^\frac{1}{2}(N^{\alpha-1}\Nmin)^{2\varepsilon}\\
& \lesssim  \sum_{\Nmin } \abrac{\Nmin}^{\beta-s} \abrac{N}^\beta N^{\frac{1-\alpha}{2}}(N^{\alpha-1}\Nmin)^{\varepsilon}\\
&\les \begin{cases} N^{2\beta +\frac{1-\alpha}{2} - s + 2\alpha\varepsilon}, \quad &{\textrm{if}}\; \beta-s\geq 0,\\
N^{\beta +\frac{1-\alpha}{2}+2\alpha\varepsilon} &{\textrm{if}}\; \beta-s< 0,
\end{cases}
\end{split}
\end{equation*}
which is again sufficient for our purpose.\\

\noi
{\it Case 2.5}.
Next, we consider the cases $L_2 = \Lmax$ and $N_3 \sim N_1 \gg N_2$. In this case, again we have $1 \lesssim L_2 \sim N^{\alpha-1}\Nmin$, hence $\Nmin \gtrsim N^{1-\alpha}$. Furthermore, in this case $L_2 \sim H \sim N^\alpha$.
By Proposition \ref{CountingEst}, summing over $\Lmax$ and $\Lmin$ first, we have
\begin{equation*}
\begin{split}
\eqref{Hsim} & \lesssim \sum_{\Nmax \sim \Nmed \sim N}  \sum_{\Lmed \geq \Lmed \gtrsim 1} \frac{\abrac{N}^\beta \abrac{\Nmin}^\beta }{\abrac{N}^s L_1^{\frac{1}{2}+\varepsilon}\Lmax^{\frac{1}{2}-2\varepsilon}} \Lmin^\frac{1}{2} \Nmin^\frac{1}{2} \\
& \lesssim \sum_{\substack{\Nmax \sim \Nmed \sim N \\  \Nmin \gtrsim N^{1-\alpha}}} \frac{\abrac{N}^\beta \abrac{\Nmin}^\beta}{\abrac{N}^s (N^{\alpha})^{\frac{1}{2}-2\varepsilon}} \Nmin^{\frac{1}{2}} \log(N) \\
&\les \begin{cases} N^{2\beta +\frac{1-\alpha}{2} - s + (2\alpha\varepsilon+\eps)}, \quad &{\textrm{if}}\; \beta\geq 0,\\
   N^{\beta + \frac{1-\alpha}{2} -s +(2\alpha\varepsilon + \eps)} + N^{2\beta + \frac{1-\alpha}{2}-s+(2\alpha\varepsilon+\eps)} &{\textrm{if}}\; \beta< 0.
   \end{cases}
\end{split}
\end{equation*}
Note again for $\beta < 0$ we have split into cases $\Nmin \leq 1$ and $\Nmin \geq 1$ to get each summand respectively. We complete the proof of this case. \\



\noi
{\it Case 2.6}.
Finally, we consider the case of $L_3 = \Lmax$ and $N_1 \sim N_2 \gg N_3$. Here again $1 \lesssim L_3 \sim H \sim N^{\alpha-1}\Nmin$ so we restrict to $\Nmin \gtrsim N^{1-\alpha}$. By Proposition \ref{CountingEst}, we have
\begin{equation}\label{Sum5.1}
\begin{split}
\eqref{Hsim} & \lesssim \sum_{\substack{\Nmax \sim \Nmed \sim N \\ \Nmin \gtrsim N^{1-\alpha}}} \sum_{\Lmed \geq \Lmin \gtrsim 1}\frac{\abrac{N}^{2\beta}  \Lmin^\frac{1}{2} \min(\Nmin^\frac{1}{2}, N^{\frac{2-\alpha}{2}}\Nmin^{-\frac{1}{2}}\Lmed^\frac{1}{2})}{\abrac{N}^s L_1^{\frac{1}{2}+\varepsilon} L_2^{\frac{1}{2}-2\varepsilon}} \\
& \lesssim \sum_{\substack{\Nmax \sim \Nmed \sim N \\ \Nmin \gtrsim N^{1-\alpha}}} \sum_{\Lmed \gtrsim 1} \frac{\abrac{N}^{2\beta}  \min(\Nmin^\frac{1}{2}, N^{\frac{2-\alpha}{2}}\Nmin^{-\frac{1}{2}}\Lmed^\frac{1}{2})}{\abrac{N}^s \Lmed^{\frac{1}{2}-2\varepsilon}}.
\end{split}
\end{equation} 

\noi
{\it Case 2.6.1}.
Firstly, we consider when $\{N^{1-\alpha} \leq \Nmin \leq N^{\frac{2- \alpha}{2}}\}$.
The sum \eqref{Sum5.1} in this case  can be estimated by
\begin{equation*}
 \eqref{Sum5.1}\lesssim\sum_{\Nmin \lesssim N^{\frac{2-\alpha}{2}}} \frac{\abrac{N}^{2\beta}}{\abrac{N}^s}\Nmin^\frac{1}{2}  \lesssim N^{2\beta + \frac{2-\alpha}{4} - s},
\end{equation*}
imposing condition $s \geq 2\beta + \frac{2 -\alpha}{4}$ to obtain a uniform bound.\\

\noi
{\it Case 2.6.2}.
 For the case when  $\{N^{\frac{2-\alpha}{2}}  < \Nmin \lesssim N\}$, we split up condition on $\Lmed$ into $\{ 1 \leq \Lmed \leq N^{\alpha-2}\Nmin^2\}$ and $\{ \Lmed > N^{\alpha-2}\Nmin^2\}$. We consider each of these cases separately. For the first case, we estimate \eqref{Sum5.1} as
\begin{equation*}
\begin{split}
\eqref{Sum5.1}& \lesssim \sum_{N^{\frac{2-\alpha}{2}} \leq \Nmin \lesssim N} \sum_{ \Lmed \lesssim N^{\alpha-2}\Nmin^2 } \frac{\abrac{N}^{2\beta}}{\abrac{N}^s} N^{\frac{2-\alpha}{2}}\Nmin^{-\frac{1}{2}}\Lmed^{2\varepsilon}  \\
& \lesssim  N^{2\beta + \frac{2-\alpha}{2} -s} N^{\frac{\alpha-2}{4}} N^{8\varepsilon}  \leq N^{2\beta +\frac{2-\alpha}{4} - s}. 
\end{split}
\end{equation*}
which is sufficient. 
While for the last case, one can bound \eqref{Sum5.1} as
\begin{equation*}
\begin{split}
\eqref{Sum5.1}& \lesssim \sum_{N^{\frac{2-\alpha}{2}} \leq \Nmin \lesssim N} \sum_{N^{\alpha-2}\Nmin^2 \lesssim \Lmed} \frac{\abrac{N}^{2\beta}}{\abrac{N}^s \Lmed^{\frac{1}{2}-2\varepsilon}} \Nmin^\frac{1}{2}  \\
& \lesssim \sum_{N^{\frac{2-\alpha}{2}} \leq \Nmin \lesssim N} \frac{\abrac{N}^{2\beta}}{\abrac{N}^s (N^{\alpha-2} \Nmin^2)^{\frac{1}{2}-2\varepsilon}} \Nmin^\frac{1}{2}  \\
& \leq N^{2\beta +\frac{2-\alpha}{4} - s} 
\end{split}
\end{equation*}
again arriving at the desired uniform bound.

Gathering all the cases, we complete the proof of the proposition.
\end{proof}

\begin{remark}\label{Bilin-T} \rm
It is worth noting that one may attempt to establish an analogous bilinear estimate in the periodic setting as well. 
When $\beta = 0$, this was carried out in \cite{CHKL-15}, following essentially the same argument as on the real line. 
However, in the periodic setting, the lack of symmetry in the multiplier when $\beta \neq 0$ introduces an additional difficulty: the analogue of \eqref{Hsim} in the case $L_3 = \Lmax$ and $N_3 = \Nmax$ becomes more challenging to control. 
In particular, when $\be > 0$,
\eqref{MMTEquation} has a derivative nonlinearity,
we do not expect that the argument here can apply in the periodic setting.
We shall develop a different approach for the periodic setting in our forthcoming paper.
\end{remark}

Now, we move on to proving the trilinear estimates \eqref{tlinear}  stated in Proposition \ref{prop-1}. 
We first address the non-negative regularity case $ s\ge 0$ and leave the negative regularity case $s < 0$ for the next section.
Our strategy is to reduce the argument to establishing the key bilinear estimates in Proposition \ref{bilinear}.

\begin{proof}[Proof of Proposition \ref{prop-1} with $s \ge 0$]
Using duality and Fourier transform properties, proving the trilinear estimate \eqref{tlinear} reduces to proving  
\begin{equation}\label{tln-e6}
\Big|\int_{\substack{\xi_1+\cdots+\xi_4=0\\ \tau_1+\cdots+\tau_4=0}} m(\xi_1,\tau_1,\cdots,\xi_4,\tau_4)\prod_{j=1}^4{f}_j(\xi_j, \tau_j)\Big| \lesssim \prod_{j=1}^4\|f_j\|_{L^2_{\xi\tau}},
\end{equation}
where
\begin{equation*}
\begin{split}
&f_1(\xi_1, \tau_1):= \langle\xi_1\rangle^s\langle\tau_1-|\xi_1|^{\alpha}\rangle^{\frac12+\epsilon}\widehat{u}(\xi_1,\tau_1)\\
&f_2(\xi_2, \tau_2):= \langle\xi_2\rangle^s\langle\tau_2+|\xi_2|^{\alpha}\rangle^{\frac12+\epsilon}\widehat{\bar{u}}(\xi_2,\tau_2)\\
&f_3(\xi_3, \tau_3):= \langle\xi_3\rangle^s\langle\tau_3-|\xi_3|^{\alpha}\rangle^{\frac12+\epsilon}\widehat{u}(\xi_3,\tau_3)\\
&f_4(\xi_4, \tau_4):= \langle\xi_4\rangle^{-s}\langle\tau_4+|\xi_4|^{\alpha}\rangle^{\frac12-2\epsilon}\widehat{\bar{v}}(\xi_4,\tau_4).
\end{split}
\end{equation*}
and the multiplier $m(\xi_1,\tau_1,\dots,\xi_4,\tau_4)$ is given by
\begin{equation}\label{tln-e7}
\frac{\langle\xi_4\rangle^{s}\,|\xi_1|^{\beta}|\xi_2|^{\beta} |\xi_3|^{\beta} |\xi_4|^{\beta}}{\langle\tau_1-|\xi_1|^{\alpha}\rangle^{\frac12+\epsilon}\langle\tau_2+|\xi_2|^{\alpha}\rangle^{\frac12+\epsilon}\langle\tau_3-|\xi_3|^{\alpha}\rangle^{\frac12+\epsilon}\langle\tau_4+|\xi_4|^{\alpha}\rangle^{\frac12-2\epsilon}\prod_{i=1}^3 \langle \xi_i \rangle^s}.
\end{equation}

We may further rewrite \eqref{tln-e6} as
\begin{equation}\label{tln-e8}
\|m\|_{[4;\R\times\R]}\lesssim 1.
\end{equation}

Note that $\xi_1+\xi_2+\xi_3+\xi_4 =0$ implies
\begin{equation*}
\langle\xi_4\rangle^s\leq\sum_{j=1}^3\langle\xi_j\rangle^s\quad\mathrm{ for}\quad s\geq 0.
\end{equation*}
Therefore, if $|\xi_2| = \max_{j=1, 2, 3} |\xi_j|$, then $\langle\xi_4\rangle^s\les \langle\xi_2\rangle^s$ and consequently, we have
\begin{equation}\label{mult-1}
\begin{split}
 m&\leq \frac{|\xi_1|^{\beta}|\xi_2|^{\beta}}{\langle\xi_1\rangle^s\langle\tau_1-|\xi_1|^{\alpha}\rangle^{\frac12+\epsilon}\langle\tau_2+|\xi_2|^{\alpha}\rangle^{\frac12-2\epsilon}}\cdot \frac{|\xi_3|^{\beta}|\xi_4|^{\beta}}{\langle\xi_3\rangle^s \langle\tau_3-|\xi_3|^{\alpha}\rangle^{\frac12+\epsilon} \langle\tau_4+|\xi_4|^{\alpha}\rangle^{\frac12-2\epsilon}}\\
       & =: \mathfrak{m}(\xi_1,\xi_2, \tau_1, \tau_2) \cdot \mathfrak{m}(\xi_3,\xi_4, \tau_3,\tau_4),
 \end{split}
\end{equation}
where we used the relation $\langle\tau_2+|\xi_2|^{\alpha}\rangle^{\frac12-2\epsilon}\leq \langle\tau_2+|\xi_2|^{\alpha}\rangle^{\frac12+\epsilon}$.

Therefore, using the comparison principle, permutation and composition properties (see respectively Lemmas 3.1, 3.3 and 3.7 in \cite{TT-01}), it is enough to bound $\| \mathfrak{m}\|_{[3;Z\times\R]}$, defined above in \eqref{mult-1}. Furthermore, if we define
\begin{equation}
\Tilde{\mathfrak{m}}(\xi_1,\xi_2,\tau_2,\tau_3) := \bigg(\frac{\abrac{\xi_1}}{|\xi_1|}\bigg)^\beta\bigg(\frac{\abrac{\xi_2}}{|\xi_2|}\bigg)^\beta \mathfrak{m}(\xi_1,\xi_2,\tau_1,\tau_2),
\end{equation}
then for $\beta > 0$, $ \mathfrak{m} \leq \Tilde{\mathfrak{m}}$. By the comparison principle, it is now sufficient to show $\| \Tilde{\mathfrak{m}} \|_{[3;  \R^2]} \lesssim 1$. With the argument from \cite{TT-01}, described earlier, $\|\Tilde{\mathfrak{m}}\|_{[3;Z\times\R]}\lesssim 1$, is implied by showing the bilinear estimate \eqref{bil-1} proved in Proposition \ref{bilinear}. When $\beta < 0$, we can also reduce to the above, see Remark \ref{SmoothingTao}, as long as we impose that $\beta  > -\frac14$.
Similarly, when either $|\xi_1|$ or $|\xi_3|$ is the largest among $|\xi_j|$, $j=1,2,3$, the argument once again reduces to the same bilinear estimates. Since these follow by essentially identical reasoning, we omit the details.
\end{proof}

\section{Trilinear Estimate and negative regularity}\label{sec-4}
In this section, we focus on proving Proposition \ref{prop-1} for negative regularity viz., the  proof of \eqref{tln-e8} for $s<0$. Since we wish to show the trilinear estimate directly, it would help to smooth out the singularities that arise when $\beta < 0$. The following allows us to replace $D_x^\beta$ with $J_x^\beta$, when bounding the $[4;\R \times \R]$-multiplier norm.
\begin{lem} \label{SmoothingEst}
Let $m$ be a $[4; \R \times \R]$-multiplier, and $\gamma < \frac14$. Then
\begin{equation}
\big\| \prod_{j=1}^4 \left( \frac{\abrac{\xi_j}}{|\xi_j|}\right)^\gamma m \big\|_{[4; \R \times \R]} \lesssim \| \widetilde{m}\|_{[4;\R \times \R]}
\end{equation}
where the implicit constant depends only on $\gamma$, and
\[ \widetilde{m}(\xi, \tau) := \sup_{\substack{\xi' = \xi + O(1) \\ \tau' = \tau + O(1)}} |m(\xi', \tau')|.\]
\end{lem}
\begin{remark} \label{SmoothingTao}
A version of Lemma \ref{SmoothingEst} for a $[3;\R \times \R]$-multiplier is of course also true.
\begin{equation}
\big\| \prod_{j=1}^3 \left( \frac{\abrac{\xi_j}}{|\xi_j|}\right)^{\gamma_j}
m \big\|_{[3; \R \times \R]} \lesssim \| \widetilde{m} \|_{[3; \R \times \R]}
\end{equation}
for $\gamma_1 + \gamma_2 + \gamma_3 < \frac{1}{2}$. This is proven as Corollary 8.2 in \cite{TT-01}. 
We remark that the restriction $\gamma < \frac14$ gives the restriction $\be > -\frac14$ in Claim \ref{claim1}.
\end{remark}

An application of Lemma \ref{SmoothingEst} now allows us to bound the $\| \cdot \|_{[4; \R\times\R]}$ norm of \eqref{tln-e7} by the norm of 
\begin{equation}\label{Tri-1}
\begin{aligned}
\widetilde{m} &(\xi_1,\tau_1,\dots,\xi_4,\tau_4) \\
& :=\frac{\langle\xi_4\rangle^{s}\,\abrac{\xi_1}^{\beta}\abrac{\xi_2}^{\beta} \abrac{\xi_3}^{\beta} \abrac{\xi_4}^{\beta}}{\langle\tau_1-|\xi_1|^{\alpha}\rangle^{\frac12+\epsilon}\langle\tau_2+|\xi_2|^{\alpha}\rangle^{\frac12+\epsilon}\langle\tau_3-|\xi_3|^{\alpha}\rangle^{\frac12+\epsilon}\langle\tau_4+|\xi_4|^{\alpha}\rangle^{\frac12-2\epsilon}\prod_{i=1}^3 \langle \xi_i \rangle^s}.
\end{aligned}
\end{equation}
provided we take $\beta  > -\frac14$. Before we start calculating this norm, we first need corresponding trilinear multiplier estimates, as we used in the bilinear estimate. For this, we will end up localizing \eqref{Tri-1} with the multiplier
\begin{equation} \label{Ymultiplier}
Y_{N_j; \Tilde{H}; L_j} = Y_{N_1, N_2, N_3, N_4; \Tilde{H}; L_1,L_2,L_3,L_4} := \chi_{|\Tilde{h}(\xi)| \sim \Tilde{H}} \prod_{j=1}^4 \chi_{|\xi_j| \sim N_j} \chi_{|\lambda_j| \sim L_j}
\end{equation}
where, as before, $\lambda_j = \tau_j - h_j(\xi_j)$, and now with $\Tilde{h}(\xi) := h_1(\xi_1) + h_2(\xi_2) + h_3(\xi_3) + h_4(\xi_4) = |\xi_1|^\alpha - |\xi_2|^\alpha + |\xi_3|^\alpha - |\xi_4|^\alpha$. Let the following be orderings of $N_j$ and $L_j$ according to size;
$\Nmax \geq \Nmean \geq \Nmed \geq \Nmin$ and $\Lmax \geq \Lmean \geq \Lmed \geq \Lmin$. Since $\xi_1 + \xi_2 + \xi_3 +\xi_4 = 0$, we have $\Nmax \sim \Nmean$. Since we need only show \eqref{tln-e9} (see below), we can assume $N_4 = \Nmin \ll \Nmed$ and $\Nmed \geq 1$. Since the role of $\xi_1$ and $\xi_3$ is completely symmetric and at least one must be of size $\Nmax$, we will always assume that $N_1 \sim \Nmax$. This leaves only two cases, either $N_2 \sim \Nmax$ or $N_3 \sim \Nmax$.

To estimate multipliers of the form \eqref{Ymultiplier},  using Corollary 3.13 in \cite{TT-01}, with a finite partition of unity, we can localize as
\begin{equation} \label{local}
\| Y_{N_j, \Tilde{H}, L_j}\|_{[4;\R\times\R]} \lesssim \big\| \chi_{|\Tilde{h}(\xi)| \sim \Tilde{H}} \chi_{|\xi_4 - \xi_4^0| \ll \Nmin} \chi_{|\lambda_4| \sim L_4} \prod_{j=1}^3 \chi_{|\xi_j -\xi_j^0| \ll \Nmed} \chi_{|\lambda_j| \sim L_j} \big\|_{[4; \R\times\R]}
\end{equation}
for some $\xi_j^0$ where $|\xi_1^0| \sim N_1, |\xi_2^0| \sim N_2, |\xi_3^0| \sim N_3, |\xi_4^0| \sim N_4$. It is important to note that the multiplier estimate appearing in \eqref{local} is crucial in our argument as it allows us to eliminate the contribution of $\Lmax$. By Lemma 3.9 in \cite{TT-01}, we have the following characteristic function estimate
\begin{equation} \label{Count}
\| \chi_A(\xi_1) \chi_B(\xi_2) \chi_C(\xi_3) \|_{[4;Z]} \lesssim |\{ (\xi_1,\xi_2) \in A \times B : \xi - \xi_1 -\xi_2 \in C\}|^\frac{1}{2}
\end{equation}
for some $\xi \in Z$ with $\xi_3 =\xi - \xi_1 -\xi_2 $. Generalizing this notion to deal with \eqref{local}, we record the following result.
\begin{lem} \label{CountingEstLem}
Let $A,B,C$ be subsets of $\R$, and order $L_1,L_2,L_3$ such that $L_1 \geq L_2 \geq L_3>0$. Then
\begin{equation} \label{ABC}
\begin{split}
\| \chi_A(\xi_1) \chi_B(\xi_2) \chi_C(\xi_3) \chi_{|\lambda_1| \sim L_1} \chi_{|\lambda_2| \sim L_2} &\chi_{|\lambda_3| \sim L_3} \|_{[4;\R\times\R]} \\ 
&\hspace{-6cm}\lesssim (L_2L_3)^\frac12 |\{ (\xi_1, \xi_2) \in A\times B : h_1(\xi_1) + h_2(\xi_2) + h_3(\xi - \xi_1 - \xi_2) = \tau + O(L_1) \}|^\frac12
\end{split}
\end{equation}
for some $\xi, \tau \in \R$.
\end{lem}
\begin{proof}
By \eqref{Count} and the Schur's test, 
the LHS of \eqref{ABC} can be estimated by
\begin{equation} 
\begin{split}
|\{(\xi_1,\tau_1), (\xi_2,\tau_2) \in \R \times \R & : \xi_1 \in A; \xi_2 \in B; \xi - \xi_1 - \xi_2 \in C; \tau_1 = h_1(\xi_1) + O(L_1);  \\
&\tau_2 = h_2(\xi_2) + O(L_2); \tau - \tau_1 - \tau_2 = h_3(\xi - \xi_1 - \xi_2) + O(L_3) \}|^\frac12.
\end{split}
\end{equation}
For fixed $\xi_1,\xi_2$, the possible values of $(\tau_1,\tau_2)$ ranges over a rectangle of area $O(L_2L_3)$, but vanishes unless $h_1(\xi_1) + h_2(\xi_2) + h_3(\xi- \xi_1 - \xi_2) = \tau + O(L_1)$. Hence, we can combine this information to achieve the result.
\end{proof}

The following relation satisfied by the resonance relationship will also be useful in our argument.

\begin{lem}\label{Reso-Lemma}
For $1 < \alpha \leq 2$. Consider the resonance relation 
\[ \Tilde{h}(\xi_1, \xi_2, \xi_3, \xi_4) = |\xi_1|^\alpha - |\xi_2|^\alpha + |\xi_3|^\alpha -|\xi_4|^\alpha,\]
under the condition $\xi_1 + \xi_2 + \xi_3 + \xi_4 = 0$. Then one has the frequency bound
\begin{equation} \label{Resonance}
|\Tilde{h}| \gtrsim |\xi_1 + \xi_2||\xi_2 + \xi_3||\xi_{\text{max}}|^{\alpha-2},
\end{equation}
where $|\xi_\text{max}| := \max\{ |\xi_1|, |\xi_2|,|\xi_3|,|\xi_4|\}$.
\end{lem}
\begin{proof}
 For simplicity, define the function $f(x) = |x|^\alpha$ for short. By the mean value theorem and $\xi_1 + \xi_2 + \xi_3 + \xi_4 = 0$
\begin{equation} \label{ResA1}
|\xi_1|^\alpha - |\xi_2|^\alpha = (\xi_3 + \xi_4)\int_0^1 f'(\xi_2 + (\xi_3 + \xi_4)\theta ) d\theta 
\end{equation}
\begin{equation} \label{ResA2}
|\xi_3|^\alpha - |\xi_4|^\alpha =  (\xi_1 + \xi_2)\int_0^1 f'(\xi_4 + (\xi_1 + \xi_2)\theta) d\theta = - (\xi_3 + \xi_4)\int_0^1 f'(-\xi_3 + (\xi_3 +  \xi_4)\theta) d\theta.
\end{equation}

Hence, combining \eqref{ResA1} and \eqref{ResA2}, we obtain the required relation \eqref{Resonance}.  
\end{proof}

We now derive the following estimates on our multiplier defined in \eqref{Ymultiplier}.

\begin{prop}[Trilinear multiplier estimates] \label{TrilinearCountingEst}
Let $\Tilde{H}, N_1,N_2,N_3,N_4,L_1,L_2,L_3,L_4$ be dyadic numbers such that $N_4 \ll N_1,N_2,N_3$. Let $1 < \alpha \leq 2$, and $\Tilde{h}(\xi) =|\xi_1|^\alpha - |\xi_2|^\alpha + |\xi_3|^\alpha -|\xi_4|^\alpha$ be the resonance relation, with $Y_{N_j; \Tilde{H}; L_j}$ be as defined in \eqref{Ymultiplier}. Then we have the following:
\begin{itemize}
\item If $\Tilde{H} \ll \Lmax \sim \Lmean$,
\begin{equation} \label{HLL}
\| Y_{N_j, \Tilde{H}, L_j}\|_{[4;\R\times\R]}  \lesssim (\Lmin \Lmed \Nmin \Nmed)^\frac12.
\end{equation}

\item If $N_1 \sim N_3 \sim \Nmax$ and $\Tilde{H} \sim \Lmax \sim L_2$ or $\Tilde{H} \sim \Lmax \sim L_4$,
\begin{equation} \label{N1N3max}
\| Y_{N_j, \Tilde{H}, L_j}\|_{[4;\R\times\R]}  \lesssim (\Lmin \Lmed \Lmean)^\frac12 \Nmin^\frac16 \Nmed^\frac{1}{6} \Nmax^{\frac23 - \frac{\alpha}{2}}.
\end{equation}

\item If $N_1 \sim N_3 \sim \Nmax$ or $N_1 \sim N_2 \sim \Nmax\gg N_3$, $\Tilde{H} \sim \Lmax$ and not as above,
\begin{equation} \label{Count3-1}
\| Y_{N_j, \Tilde{H}, L_j}\|_{[4;\R\times\R]}  \lesssim (\Lmin \Lmed \Lmean)^\frac12 \Nmin^\frac{1}{4}\Nmax^{\frac{3}{4}-\frac{\alpha}{2}}.
\end{equation}

\end{itemize}
\end{prop}

\begin{proof}  We provide a proof of this result dividing into different cases.\\

\noi
{\bf Case 1}. First consider the case $\Tilde{H} \ll \Lmax \sim \Lmean$. In this case to show \eqref{HLL}, we use the Comparison principle and Lemma 3.6 in \cite{TT-01}  to obtain from \eqref{local} that
\begin{equation}
\| Y_{N_j, \Tilde{H}, L_j}\|_{[4;\R\times\R]} \lesssim \big\| \|\chi_{|\lambda_\text{min}| \sim \Lmin} \chi_{|\lambda_\text{med}| \sim \Lmed} \|_{[4; \R]} \chi_{|\xi_4 - \xi_4^0| \ll \Nmin} \chi_{|\xi_3 - \xi_3^0| \ll \Nmed|} \|_{[4; \R]}.
\end{equation}
By the following calculations
\[ \|\chi_{|\lambda_\text{min}| \sim \Lmin} \chi_{|\lambda_\text{med}| \sim \Lmed} \|_{[4; \R]} \lesssim \|\chi_{|\lambda_\text{min}| \sim \Lmin} \|_{[3; \R]}  \|\chi_{|\lambda_\text{med}| \sim \Lmed} \|_{[3; \R]}  \lesssim (\Lmin\Lmed)^\frac{1}{2}\]
and 
\[ \| \chi_{|\xi_4 - \xi_4^0| \ll \Nmin} \chi_{|\xi_3 - \xi_3^0| \ll \Nmed|} \|_{[4; \R]} \lesssim \| \chi_{|\xi_4 - \xi_4^0| \ll \Nmin} \|_{[3;\R]} \| \chi_{|\xi_3 - \xi_3^0| \ll \Nmed} \|_{[3;\R]} \lesssim (\Nmin \Nmed)^\frac12.\]
we observe that 
\begin{equation} \label{BaseEst}
\| Y_{N_j, \Tilde{H}, L_j}\|_{[4;\R\times\R]}  \lesssim (\Lmin \Lmed \Nmin \Nmed)^\frac12
\end{equation}
always holds, regardless of case, which demonstrates \eqref{HLL}.\\

\noi
{\bf Case 2}. If $N_1 \sim N_3 \sim \Nmax$ and $\Tilde{H} \sim \Lmax \sim L_2$ or $\Tilde{H} \sim \Lmax \sim L_4$, we prove \eqref{N1N3max} by considering each of the two subcases in turn. \\

\noi
{\it Case 2.1}. If $N_1 \sim N_3$ and $\Tilde{H} \sim \Lmax \sim L_2$, we use the Comparison principle (Lemma 3.1 in \cite{TT-01}) and Lemma \ref{CountingEstLem}, to obtain from \eqref{local} that
\begin{equation} \label{Count6Pf}
\begin{split}
\| Y_{N_j, \Tilde{H}, L_j}&\|_{[4;\R\times\R]} \\
&\lesssim (\Lmin\Lmed)^\frac12 |\{ \xi_4,\xi_3 \in \R : |\xi_4 - \xi_4^0| \ll \Nmin; |\xi_3 - \xi_3^0| \ll \Nmed; 
 \\
 &\hspace{4cm}|\xi - \xi_3 - \xi_4|^\alpha + |\xi_3|^\alpha - |\xi_4|^\alpha = \tau + O(\Lmean) \}|^\frac12,
\end{split}
\end{equation}
for some $\tau \in \R $ and $\xi$ such that $|\xi + \xi_2^0| \ll \Nmed$. 

Defining $f(\xi_3,\xi_4) = |\xi - \xi_3 - \xi_4|^\alpha + |\xi_3|^\alpha - |\xi_4|^\alpha$, when considering the interval over which $\xi_3$ ranges, we divide into two cases
\begin{equation}
    |\xi - \xi_4 - 2\xi_3| \leq R \hspace{1cm}\text{ and } \hspace{1cm}|\xi - \xi_4 - 2\xi_3| \geq R.
\end{equation}
For the former case, clearly for fixed $\xi_4$, $\xi_3$ ranges over an interval of length $\les R$. For the latter case
\begin{align*}
|\del_{\xi_3} f| &= \alpha |\xi_3|\xi_3|^{\alpha-2} + (\xi_3 + \xi_4 -\xi)|\xi_3+\xi_4- \xi|^{\alpha-2}| \\
& \sim |\xi - \xi_4 - 2\xi_3| \int_0^1 |\xi_3 + \mu(\xi - \xi_4 - 2\xi_3)|^{\alpha-2} d \mu \gtrsim R\Nmax^{\alpha-2}.
\end{align*} 
 Therefore, by the Mean Value Theorem $\xi_3$ ranges over an interval of length $\Lmean \Nmax^{2-\alpha}R^{-1}$. Hence, choosing $R = \Lmean \Nmax^{2-\alpha}R^{-1}$ gives that $\xi_3$ ranges over an interval of length $\Lmean^\frac12 \Nmax^{\frac{2-\alpha}{2}}$ in both cases.
Furthermore,
\begin{align*}
|\del_{\xi_4} f| &= \alpha |-\xi_4|\xi_4|^{\alpha-2} + (\xi_3 + \xi_4 -\xi)|\xi_3+\xi_4- \xi|^{\alpha-2}| \\
& \sim |\xi_3-\xi| \int_0^1 |\xi_4 + \mu(\xi_3 -\xi)|^{\alpha-2} d \mu \gtrsim \Nmax^{\alpha-1}.
\end{align*} 
Above can be observed from the reduction we did in \eqref{local}, as $|\xi_3 - \xi| = |\xi_3 + \xi_2^0 - (\xi + \xi_2^0)|$ of which $|\xi_3 + \xi_2^0| \sim \Nmax$ and $|\xi + \xi_2^0| \ll \Nmed$. Hence, again by Mean Value Theorem, $\xi_4$ ranges over an interval of length $\Lmean \Nmax^{1-\alpha}$. Combining all this information, \eqref{Count6Pf} yields
\begin{equation} \label{Count6Int1}
\| Y_{N_j, \Tilde{H}, L_j}\|_{[4;\R^2]} \lesssim (\Lmin\Lmed)^\frac12 \Lmean^\frac{3}{4}\Nmax^{1-\frac{3}{4}\alpha}.
\end{equation}
Interpolating between \eqref{Count6Int1} and $\eqref{BaseEst}$ then yields \eqref{N1N3max}.\\

\noi
{\it Case 2.2}. If $N_1 \sim N_3 \sim \Nmax$ and $\Tilde{H} \sim \Lmax \sim L_4$, we can obtain the required estimate again with a similar argument used in {\it Case 2.1} by considering the measure of the set
\begin{equation}
\{\xi_2,\xi_3 \in \R : |\xi_2 - \xi_2^0| \ll \Nmed, |\xi_3 - \xi_3^0| \ll \Nmed; |\xi - \xi_2 -\xi_3|^\alpha - |\xi_2|^\alpha + |\xi_3|^\alpha  = \tau + O(\Lmean)\}
\end{equation}
for some $\tau \in \R$ and $\xi \in \R$ such that $|\xi + \xi_4^0| \ll \Nmin$. \\


\noi
{\bf Case 3}. For all remaining cases such that $\Tilde{H} \sim \Lmax$, we show \eqref{Count3-1} by subdividing into several subcases.\\

\noi 
{\it Case 3.1}. If $\Tilde{H} \sim \Lmax \sim L_1$ and $N_1 \sim N_2 \sim \Nmax \gg \Nmed \sim N_3$ we use the Comparison principle and Lemma \ref{CountingEstLem} to obtain from \eqref{local} that
\begin{equation} \label{Count1Pf}
\begin{split}
\| Y_{N_j, \Tilde{H}, L_j}&\|_{[4;\R\times\R]} \\
&\lesssim (\Lmin\Lmed)^\frac12 |\{ \xi_4,\xi_3 \in \R : |\xi_4 - \xi_4^0| \ll \Nmin; |\xi_3 - \xi_3^0| \ll \Nmed; 
 \\
 &\hspace{4cm}|\xi - \xi_3 - \xi_4|^\alpha - |\xi_3|^\alpha + |\xi_4|^\alpha = \tau + O(\Lmean) \}|^\frac12,
\end{split}
\end{equation}
for some $\tau \in \R$ and $\xi$ such that $|\xi + \xi_1^0| \ll \Nmed$.

Defining $f(\xi_3,\xi_4) = |\xi - \xi_3 -\xi_4 |^\alpha - |\xi_3|^\alpha + |\xi_4|^\alpha$, we have
\begin{align*}
|\del_{\xi_3} f | &= |\alpha((\xi_3 + \xi_4 -\xi)|\xi - \xi_3 - \xi_4|^{\alpha-2} - \xi_3|\xi_3|^{\alpha-2})| \\
& \sim |\xi - \xi_4| \int_0^1 |\xi_3+ \mu(\xi_4 - \xi)|^{\alpha-2} d \mu \gtrsim \Nmax^{\alpha-1},
\end{align*}
for $|\xi_4| \sim \Nmin \ll \Nmax \sim |\xi_1^0|$. Hence, by the Mean Value Theorem,
$\xi_3$ ranges over an interval of length $\Lmean \Nmax^{1-\alpha}$. Clearly, $\xi_4$ ranges over a length of $\Nmin$, hence overall, the RHS of \eqref{Count1Pf} can be bounded by $(\Lmin\Lmed\Lmean)^\frac12 \Nmin^\frac12 \Nmax^{\frac{1-\alpha}{2}}$, which is sufficient for \eqref{Count3-1}.\\

\noi
{\it Case 3.2}.  If $\Tilde{H} \sim \Lmax \sim L_2$ and $N_1 \sim N_2 \sim \Nmax\gg \Nmed \sim N_3$, one can obtain \eqref{Count3-1} similarly as in {\it Case 3.1} by considering the measure of the set
\begin{equation}
\{\xi_3,\xi_4 \in \R: |\xi_3 - \xi_3^0| \ll \Nmed, |\xi_4 - \xi_4^0| \ll \Nmin; |\xi - \xi_3 - \xi_4|^\alpha + |\xi_3|^\alpha - |\xi_4|^\alpha = \tau + 
O(\Lmean)\},
\end{equation}
for some $\tau \in \R$ and $\xi \in \R$ such that $|\xi + \xi_2^0|\ll \Nmed$. \\
  

\noi
{\it Case 3.3}.  If $\Tilde{H} \sim \Lmax \sim L_3$ and $N_1 \sim N_2 \sim \Nmax\gg \Nmed \sim N_3$, similarly to previous cases, using Comparison principle and Lemma \ref{CountingEstLem},  one obtains from \eqref{local} that
\begin{equation} \label{Count3Pf}
\begin{split}
\| Y_{N_j, \Tilde{H}, L_j}&\|_{[4;\R\times\R]} \\
&\lesssim (\Lmin\Lmed)^\frac12 |\{ \xi_4,\xi_2 \in \R : |\xi_4 - \xi_4^0| \ll \Nmin; |\xi_2 - \xi_2^0| \ll \Nmed; 
 \\
 &\hspace{4cm}|\xi - \xi_2 - \xi_4|^\alpha - |\xi_2|^\alpha 
 - |\xi_4|^\alpha = \tau + O(\Lmean) \}|^\frac12,
\end{split}
\end{equation}
for some $\tau \in \R$ and $\xi$ where $|\xi + \xi_3^0| \ll \Nmed$. 

Defining $f(\xi_2,\xi_4) = |\xi - \xi_2 - \xi_4|^\alpha - |\xi_2|^\alpha - |\xi_4|^\alpha$, one gets
\begin{align*}
|\del_{\xi_2} f| & = \alpha |(\xi_2 + \xi_4 -\xi)|\xi_2 + \xi_4 -\xi|^{\alpha-2} - \xi_2|\xi_2|^{\alpha-2}| \\
& \sim |\xi_4- \xi|\int_0^1 |\xi_2 + \mu(\xi_4- \xi)|^{\alpha-2} d \mu \gtrsim \Nmed\Nmax^{\alpha-2}.
\end{align*}
Thus $\xi_2$ ranges over an interval of length $\Lmean \Nmed^{-1} \Nmax^{2-\alpha}$. Furthermore,
\begin{align*}
|\del_{\xi_4} f| & = \alpha |(\xi_2 + \xi_4 -\xi)|\xi_2 + \xi_4 -\xi|^{\alpha-2} - \xi_4|\xi_4|^{\alpha-2}| \\
& \sim |\xi_2 - \xi|\int_0^1 |\xi_4 + \mu(\xi_2- \xi)|^{\alpha-2} d \mu \gtrsim \Nmax^{\alpha-1},
\end{align*}
since $|\xi_2| \sim \Nmax \gg \Nmed \sim |\xi_3^0|$, hence $\xi_4$ ranges over an interval of length $\Lmean\Nmax^{1-\alpha}$. 

Combining all information, \eqref{Count3Pf} yields the estimate
\begin{equation} \label{Count3Int1}
\| Y_{N_j, \Tilde{H}, L_j}\|_{[4;\R\times\R]} \lesssim (\Lmin\Lmed)^\frac12 \Lmean \Nmed^{-\frac{1}{2}}\Nmax^{\frac{3}{2}-\alpha}.
\end{equation}
Finally, interpolating between \eqref{Count3Int1} and \eqref{BaseEst}, yields the required \eqref{Count3-1}.\\

\noi
{\it Case 3.4}. If $\Tilde{H} \sim \Lmax \sim L_4$ and $N_1 \sim N_2 \sim \Nmax\gg \Nmed \sim N_3$, one can obtain \eqref{Count3-1} similarly as in {\it Case 3.3} 
by considering the measure of the set
\begin{equation}
\{\xi_1,\xi_3 \in \R: |\xi_1 - \xi_1^0| \ll \Nmed, |\xi_3 - \xi_3^0| \ll \Nmed; |\xi_1|^\alpha - |\xi - \xi_1 - \xi_3|^\alpha + |\xi_3|^\alpha = \tau + O(\Lmean)\},
\end{equation}
for some $\tau \in \R$ and $\xi \in \R$ such that $|\xi + \xi_4^0| \ll \Nmin$. \\


\noi
{\it Case 3.5}. If $\Tilde{H} \sim \Lmax \sim L_1$ and $N_1 \sim N_3 \sim \Nmax$, one can obtain \eqref{Count3-1} similarly as in {\it Case 3.1} by considering the measure of the set
\begin{equation}
\{\xi_2,\xi_4 \in\R: |\xi_2 - \xi_2^0| \ll \Nmed, |\xi_4 - \xi_4^0| \ll \Nmin; |\xi_2|^\alpha - |\xi - \xi_2 - \xi_4|^\alpha + |\xi_4|^\alpha = \tau + O(\Lmean) \},
\end{equation}
for some $\tau \in \R$ and $\xi \in \R$ such that $|\xi + \xi_1^0| \ll \Nmed$. \\


\noi{\it Case 3.6}. If $\Tilde{H} \sim \Lmax \sim L_3$ and $N_1 \sim N_3 \sim \Nmax$, one can obtain \eqref{Count3-1} similarly as in {\it Case 3.1} by considering the measure of the set
\begin{equation}
\{ \xi_2, \xi_4 \in \R: |\xi_2 - \xi_2^0| \ll \Nmed, |\xi_4 - \xi_4^0| \ll \Nmin; |\xi - \xi_2 - \xi_4|^\alpha - |\xi_2|^\alpha - |\xi_4|^\alpha = \tau +O(\Lmean) \},
\end{equation}
for some $\tau \in \R$ and $\xi \in \R$ such that $|\xi + \xi_3^0| \ll \Nmed$.
\end{proof}

Now we are ready to state the main lemma of this section.

\begin{lemma} \label{claim1}
Let $m$ be defined as in \eqref{tln-e7}, then for $\beta \in (-\frac14, 0)$ and $0 > s > \beta + \frac{5}{8} - \frac{\alpha}{2}$ we have
\begin{equation} \label{tln-e9}
\|m(\xi_1,\tau_1,\dots, \xi_4,\tau_4) \chi_{|\xi_4| = |\xi_\text{min}| \ll |\xi_\text{med}| \lesssim |\xi_\text{mean}|\sim |\xi_\text{max}|} \chi_{|\xi_\text{med}| \geq 1} \|_{[4;\R\times\R]} \lesssim 1.
\end{equation} 
\end{lemma}

\begin{proof}[Proof of Lemma \ref{claim1}]
By an averaging argument, we may restrict to $|\lambda_j| \geq 1$. By a dyadic decomposition and Schur's test (Lemma 3.11 in \cite{TT-01}), for $\Nmax \sim N$,  $\|m(\xi_1,\tau_1,\dots, \xi_4,\tau_4) \chi_{|\xi_4| = |\xi_\text{min}| \ll |\xi_\text{med}| \lesssim |\xi_\text{mean}|\sim |\xi_\text{max}|} \chi_{|\xi_\text{med}| \geq 1} \|_{[4;\R\times\R]}$ is bounded by
\begin{equation} \label{TrilinearToBound}
\begin{split}
 \sup_N \sum_{\substack{\Nmin \ll \Nmed \leq N \\ \Nmed \geq 1}} \sum_{L_j} \sum_{\Tilde{H}} \frac{\abrac{N}^{2\beta -2s} \abrac{\Nmed}^{\beta - s} \abrac{\Nmin}^{\beta + s}}{L_1^{\frac12 + \varepsilon}L_2^{\frac12 + \varepsilon}L_3^{\frac12 + \varepsilon}L_4^{\frac12 - 2\varepsilon}} \| Y_{N_j, \Tilde{H}, L_j}\|_{[4;\R^2]}.
\end{split}
\end{equation}
Note that by the resonance relation \eqref{Resonance} one has 
\begin{equation}\label{ResonanceGeneral}
    \Tilde{H} \gtrsim N^{\alpha-1}\Nmed.
\end{equation}
\noi
{\bf Case 1}.
Consider first the case for which $\Tilde{H} \ll \Lmax \sim \Lmean$. By \eqref{HLL}, we have
\begin{align}
\eqref{TrilinearToBound} & \lesssim  \sum_{\substack{\Nmin \ll \Nmed \leq N \\ \Nmed \geq 1}} \sum_{1 \leq \Lmin\leq \Lmed \leq \Lmax} \frac{\abrac{N}^{2\beta -2s} \abrac{\Nmed}^{\beta - s} \abrac{\Nmin}^{\beta + s}}{\Lmin^{\frac12 + \varepsilon} \Lmed^{\frac12 + \varepsilon} \Lmax^{1 - \varepsilon}} (\Lmin\Lmed\Nmin\Nmed)^\frac{1}{2} \nonumber\\
& \lesssim \sum_{\substack{\Nmin \ll \Nmed \leq N \\ \Nmed \geq 1}} \abrac{N}^{2\beta -2s} \abrac{\Nmed}^{\frac12 + \beta - s} \abrac{\Nmin}^{\beta + s}\Nmin^\frac12  N^{(1-\alpha) + \varepsilon(\alpha-1)} \Nmed^{-1 + \varepsilon} \nonumber\\
& \lesssim \sum_{\substack{\Nmin \ll \Nmed \leq N \\ \Nmed \geq 1}} \abrac{N}^{2\beta -2s + (1- \alpha) +\varepsilon(\alpha-1)} \abrac{\Nmed}^{\beta - \frac12 - s + \varepsilon} \abrac{\Nmin}^{\beta + s}\Nmin^\frac12.   \label{Bound1LHS}
\end{align}

For $\Nmin \leq 1$, we can bound this as
\begin{equation} \label{TriCompute1}
\eqref{Bound1LHS} \lesssim \sum_{1 \leq  \Nmed \leq N } \abrac{N}^{2\beta -2s + (1- \alpha) +\varepsilon(\alpha-1)} \abrac{\Nmed}^{\beta - \frac12 - s + \varepsilon} \leq N^{2\beta - 2s + (1-\alpha) + \varepsilon(\alpha-1)}
\end{equation}
since $\beta - \frac{1}{2} - s < 0$ for $s,\beta$ in the range considered. Observe that \eqref{TriCompute1} is bounded for $s > \beta + \frac{1-\alpha}{2}$. 

For $\Nmin \geq 1$, if $\beta + s + \frac12 < 0$, we can follow  the computation as above. Otherwise 
\begin{equation}
\eqref{Bound1LHS} \lesssim \sum_{1 \leq  \Nmed \leq N } \abrac{N}^{2\beta -2s + (1- \alpha) +\varepsilon(\alpha-1)} \abrac{\Nmed}^{2\beta + \varepsilon} \leq N^{2\beta - 2s + (1-\alpha) + \varepsilon(\alpha-1)},
\end{equation}
giving us the same bound as before.\\

\noi
{\bf Case 2}. We consider the second case for which $N_1 \sim N_3 \sim N$ and $\Tilde{H} \sim \Lmax \sim L_2$ or $\Tilde{H} \sim \Lmax \sim L_4$. Observe that due to the restriction $N_1 \sim N_3 \sim N$, by \eqref{Resonance} we have an improvement on \eqref{ResonanceGeneral} with $\Tilde{H} \gtrsim N^\alpha$. By \eqref{N1N3max}, we have the following bound on \eqref{TrilinearToBound}.
\begin{align}
& \sum_{\substack{\Nmin \ll \Nmed \leq N \\ \Nmed \geq 1}} \sum_{L_j: j \in \{1,2,3,4\}} \frac{\abrac{N}^{2\beta -2s} \abrac{\Nmed}^{\beta - s} \abrac{\Nmin}^{\beta + s}}{\Lmin^{\frac12 + \varepsilon}\Lmed^{\frac12 + \varepsilon} \Lmean^{\frac12 + \varepsilon} \Lmax^{\frac12 - 2\varepsilon}} (\Lmin\Lmed\Lmean)^\frac12 \Nmin^\frac16 \Nmed^\frac16 N^{\frac{2}{3}-\frac{\alpha}{2}} \\
& \les \sum_{\substack{\Nmin \ll \Nmed \leq N \\ \Nmed \geq 1}} \abrac{N}^{2\beta-2s +\frac{2}{3} - \frac{\alpha}{2}}\abrac{\Nmed}^{\beta-s}\abrac{\Nmin}^{\beta+s}\Nmin^\frac16 \Nmed^\frac16 N^{-\frac{\alpha}{2} +2\alpha\varepsilon} \\
& = \sum_{\substack{\Nmin \ll \Nmed \leq N \\ \Nmed \geq 1}} \abrac{N}^{2\beta-2s +\frac{2}{3} - \alpha + 2\alpha\varepsilon}\abrac{\Nmed}^{\beta-s + \frac16}\abrac{\Nmin}^{\beta+s}\Nmin^\frac16. \label{Bound2LHS}
\end{align}

For $\Nmin \leq 1$, we get
\begin{equation}
\eqref{Bound2LHS} \lesssim \sum_{1 \leq \Nmed \leq N} \abrac{N}^{2\beta - 2s + \frac23 - \alpha + 2\alpha\varepsilon} \abrac{\Nmed}^{\beta-s+\frac{1}{6}}.\label{Bound2RHS}
\end{equation}
If $\beta -s + \frac16 < 0$ then we get a uniform bound for $s > \beta +\frac{1}{3} - \frac{\alpha}{2}$. Otherwise 
\begin{equation}
    \eqref{Bound2RHS} \les N^{3\beta - 3s + \frac{5}{6} - \alpha + 2\alpha\varepsilon}
\end{equation}
which is uniformly bounded for $s > \beta + \frac{5}{18} - \frac{\alpha}{3}$. Note that this is implied by $s > \beta + \frac{5}{8} - \frac{\alpha}{2}$ in the ranges we consider. 

For $\Nmin \geq 1$, if $\beta + s + \frac{1}{6} < 0$, we reduce to the estimates above. Otherwise, we bound by
\begin{equation}
\eqref{Bound2LHS} \les \sum_{1 \leq \Nmed \leq N} \abrac{N}^{2\beta - 2s + \frac{2}{3} - \alpha + 2\alpha\varepsilon} \abrac{\Nmed}^{2\beta + \frac{1}{3}}.
\end{equation}
If $2\beta + \frac13 < 0$ then we have a uniform bound for $s > \beta +\frac{1}{3} - \frac{\alpha}{2}$. Otherwise, we have a uniform bound for $s > 2\beta + \frac{1-\alpha}{2}$, which is again sufficient, completing the case. \\

\noi
{\bf Case 3}. We now consider the remaining cases for $\Lmax \sim \Tilde{H}$. Taking into consideration the estimate \eqref{Count3-1} from Proposition \ref{TrilinearCountingEst}, it is sufficient to estimate the following
\begin{align}
&\sum_{\substack{\Nmin \ll \Nmed \leq N \\ \Nmed \geq 1}} \sum_{L_j: j \in\{1,2,3,4\}} \frac{\abrac{N}^{2\beta -2s} \abrac{\Nmed}^{\beta - s} \abrac{\Nmin}^{\beta + s}}{\Lmin^{\frac12 + \varepsilon}\Lmed^{\frac12 + \varepsilon} \Lmean^{\frac12 + \varepsilon} \Lmax^{\frac12 - 2\varepsilon}} (\Lmin\Lmed\Lmean)^\frac12 \Nmin^\frac14 N^{\frac{3}{4} - \frac{\alpha}{2}} \nonumber \\
& \lesssim \sum_{\substack{\Nmin \ll \Nmed \leq N \\ \Nmed \geq 1}} \abrac{N}^{2\beta -2s + \frac34 - \frac{\alpha}{2}} \abrac{\Nmed}^{\beta - s} \abrac{\Nmin}^{\beta + s}\Nmin^\frac{1}{4} N^{\frac{1-\alpha}{2} +2(\alpha-1)\varepsilon}\Nmed^{-\frac12 + 2\varepsilon} \nonumber \\
& \lesssim \sum_{\substack{\Nmin \ll \Nmed \leq N \\ \Nmed \geq 1}}\abrac{N}^{2\beta -2s + \frac54 - \alpha + 2(\alpha-1)\varepsilon} \abrac{\Nmed}^{\beta - s - \frac{1}{2} + 2\varepsilon} \abrac{\Nmin}^{\beta + s}\Nmin^\frac{1}{4} .\label{Bound3LHS}
\end{align}

For $\Nmin \leq 1$, we get
\begin{equation}
\eqref{Bound3LHS} \lesssim \sum_{1 \leq \Nmed \leq N} \abrac{N}^{2\beta -2s + \frac54 - \alpha + 2(\alpha-1)\varepsilon} \abrac{\Nmed}^{\beta - s - \frac{1}{2} + 2\varepsilon} \lesssim N^{2\beta - 2s + \frac54 - \alpha + 2(\alpha-1)\varepsilon},
\end{equation}
again where $\beta - s - \frac{1}{2} < 0$, giving a uniform bound for $s > \beta + \frac{5}{8} - \frac{\alpha}{2}$.

For $\Nmin \geq 1$, again if $\beta + s + \frac{1}{4} < 0$, we reduce to the estimates above. Otherwise, we can bound by
\begin{equation}
\eqref{Bound3LHS} \lesssim 
 \sum_{1 \leq \Nmed \leq N} \abrac{N}^{2\beta -2s + \frac54 - \alpha + 2(\alpha-1)\varepsilon} \abrac{\Nmed}^{2\beta - \frac{1}{4} + 2\varepsilon} \lesssim N^{2\beta - 2s + \frac54 - \alpha + 2(\alpha-1)\varepsilon},
\end{equation}
leaving the same requirement on $s,\beta, \alpha$ and finishing the proof.
\end{proof}

Now, we are in a position to provide a proof of the trilinear estimate for negative regularities.

\begin{proof}[Proof of Proposition \ref{prop-1} for $s < 0$]
Returning to the context of the proof of Proposition \ref{prop-1}, in the case where $s<0$ we cannot use the original argument to cancel the $\abrac{\xi_4}$ term. Instead we can perform a ``half-cancellation''. We may assume that $|\xi_4| \ll |\xi_1|, |\xi_2|, |\xi_3|$ as otherwise we can perform a full cancellation as in \eqref{mult-1}. Since $\xi_1 + \xi_2 + \xi_3 + \xi_4 = 0$, we have three cases to consider $|\xi_2| \lesssim |\xi_1| \sim |\xi_3|$, $|\xi_1| \lesssim |\xi_3| \sim |\xi_2|$ and $|\xi_3| \lesssim |\xi_2| \sim |\xi_1|$. In the first case we can observe
\begin{equation}
\frac{\abrac{\xi_1}^{-s}\abrac{\xi_2}^{-s}\abrac{\xi_3}^{-s}}{\abrac{\xi_4}^{-s}} \lesssim \abrac{\xi_1}^{-\frac32s} \abrac{\xi_3}^{-\frac32s}.
\end{equation}
Arguing again, by the same bilinear estimate reduction, we can recover the trilinear estimate for negative regularity, but now for $\frac32s \geq 2\beta + \frac{2-\alpha}{4}$ instead. The cases $|\xi_1| \lesssim |\xi_3| \sim |\xi_2|$ and $|\xi_3| \lesssim |\xi_2| \sim |\xi_1|$ follow in a completely analogous fashion.

However, for $\beta \in (\frac78-\alpha,0)$, 
we can improve on this by using Lemma \ref{claim1}.
To see this, we may assume $|\xi_4| \ll |\xi_1|,|\xi_2|,|\xi_3|$ and $|\xi_i| \sim |\xi_k| \sim \max_{j=1,2,3} |\xi_j|$ for $i \neq k$, $i,k \in\{1,2,3\}$. If we order indices in terms of size $|\xi_{\text{max}}| \geq |\xi_{\text{mean}}| \geq |\xi_{\text{med}}| \geq |\xi_{\text{min}}|$, then we can reduce to the bilinear estimate as before in all cases other than $|\xi_4| = |\xi_{\text{min}}| \ll |\xi_{\text{med}}| \lesssim |\xi_{\text{mean}}| \sim |\xi_{\text{max}}|$. Note further, if $|\xi_\text{med}| \leq 1$ then $\abrac{\xi_\text{min}}^s \sim \abrac{\xi_\text{med}}^s \sim 1$, in which case we can again reduce to the bilinear estimate.
We may therefore also assume that $|\xi_\text{med}| > 1$ in the following. 
Therefore, it only remains to prove the multiplier estimate \eqref{tln-e8} under the conditions
\[
\{|\xi_4| = |\xi_\text{min}| \ll |\xi_\text{med}| \lesssim |\xi_\text{mean}|\sim |\xi_\text{max}| \} \cap \{|\xi_\text{med}| \geq 1 \} ,
\]
which follows from Lemma \ref{claim1}. This completes the proof.
\end{proof}

\section{Ill-posedness and optimal parameter thresholds}\label{sec-5}

In this section, we prove the ill-posedness result stated in Theorem~\ref{Th-ill-Line}.
Recall that the theorem asserts that, outside the range
\[
 s \;\ge\; 2\beta + \frac{2-\alpha}{4}, \qquad -\frac14 \;\le\; \beta \;\le\; \frac{\alpha-1}{2}, 
\]
there cannot exist a local solution theory for the IVP~\eqref{MMTEquation} whose
data-to-solution map is $\mathcal C^3$ at the origin in $H^s(\R)$.

More precisely, we will show that if $s,\alpha,\beta$ violate any of the three conditions
\[
 s \;\ge\; 2\beta + \frac{2-\alpha}{4}, \qquad \beta \;\ge\; -\frac14, \qquad \beta \;\le\; \frac{\alpha-1}{2},
\]
then the flow map
\[
 \Phi_t : H^s(\R)\to H^s(\R),\qquad \Phi_t(u_0)=u(t),
\]
cannot be $\mathcal C^3$ at the origin. Here $u$ is a (hypothetical) solution to the IVP~\eqref{MMTEquation}
with $u(0)=u_0$. As usual, the notion of $\mathcal C^3$ is that of Fr\'echet differentiability.
In particular, this rules out the possibility of constructing a local solution by a standard
contraction mapping argument in $H^s$; such a scheme would necessarily produce a smooth
data-to-solution map.

\subsection{Strategy of the proof}

Let $u_0\in H^s(\R)$ and $\delta>0$ be a small parameter. We consider the IVP
\begin{equation*}
 \begin{cases}
  i\partial_t u + (-\partial_x^2)^{\alpha/2} u
    = D_x^\beta \bigl(|D_x^\beta u|^2 D_x^\beta u\bigr),\qquad x,t\in\R,\\[1mm]
  u(x,0) = \delta\,u_0(x),
 \end{cases}
\end{equation*}
and denote the corresponding solution (when it exists) by $u=u^\delta(t)$.
Formally, we may write $u$ using Duhamel's formula as
\begin{equation}\label{C3-Duhamel}
 u^\delta(t)
   = \delta\,S(t)u_0
     - i \int_0^t S(t-t_1)\, D_x^\beta\bigl(|D_x^\beta u^\delta(t_1)|^2 D_x^\beta u^\delta(t_1)\bigr)\,dt_1,
\end{equation}
where $S(t)$ is the linear propagator
\[
 S(t)u_0 := \bigl(e^{it|\xi|^\alpha}\widehat{u_0}(\xi)\bigr)^\vee.
\]

Assume, for contradiction, that for some $s\in\R$ and $T>0$ small, the IVP~\eqref{MMTEquation}
is locally well-posed in $H^s(\R)$ and the solution map
\[
 \Phi_t : H^s(\R)\to H^s(\R),\qquad \Phi_t(u_0)=u(t),
\]
is $\mathcal C^3$ at the origin for all $t\in[0,T]$. In particular, the mapping
\[
 u_0 \mapsto u^\delta(t)
\]
is $\mathcal C^3$ in $\delta$ near $\delta=0$ with values in $H^s(\R)$.

Differentiating~\eqref{C3-Duhamel} in $\delta$ at $\delta=0$ (or, equivalently,
expanding $u$ in powers of $\delta$ and using the cubic nature of the nonlinearity)
shows that the third Fr\'echet derivative $d_0^3\Phi_t$ at the origin is given by
\begin{equation}\label{C3-d3Phi}
 d_0^3\Phi_t(u_0,u_0,u_0)
  = -i \int_0^t S(t-t_1)\,
      D_x^\beta\bigl(
        S(t_1)D_x^\beta u_0 \cdot \overline{S(t_1)D_x^\beta u_0} \cdot S(t_1)D_x^\beta u_0
      \bigr)\,dt_1.
\end{equation}
Since the nonlinearity is trilinear in $D_x^{\beta}u$, the first and second derivatives vanish at
the origin and $d^3_0\Phi_t$ is the first non-trivial derivative.

If $\Phi_t$ is $\mathcal C^3$ at $0$, there exists a constant $C_t>0$ such that
\begin{equation}\label{C3-bound}
 \bigl\|d_0^3\Phi_t(u_0,u_0,u_0)\bigr\|_{H^s}
   \le C_t \,\|u_0\|_{H^s}^3
 \qquad \text{for all } u_0\in H^s(\R).
\end{equation}
Our aim is to construct carefully chosen high- or low-frequency initial data $u_0$ for which
this estimate fails, depending on the parameters $(s,\alpha,\beta)$. We will do this by working
in Fourier space and localising $u_0$ to suitable frequency boxes. The analysis splits into
three different frequency configurations, each of which produces a different constraint on $s$
and $\beta$.

\subsection{Fourier representation}

Let us compute the Fourier transform in $x$ of the quantity in~\eqref{C3-d3Phi}. Denote
\[
 \omega(\xi) := |\xi|^\alpha,
\]
and recall 
\[
 \widehat{D_x^\beta f}(\xi) = |\xi|^\beta \widehat{f}(\xi) \quad\mathrm{and}\quad \widehat{S(t)f}(\xi) = e^{it\omega(\xi)}\widehat{f}(\xi).
\]
Using this and applying the Fourier transform to~\eqref{C3-d3Phi}, we obtain
\begin{equation}\label{Est-F1}
\begin{split}
 & \widehat{d_0^3  \Phi_t(u_0,u_0,u_0)}(\xi) \\
 & = i\,|\xi|^\beta e^{it\omega(\xi)}
    \int_{\xi_1-\xi_2+\xi_3=\xi} 
      \frac{e^{it \Omega(\vec{\xi})}-1}{\Omega(\vec{\xi})}\,
      |\xi_1|^\beta \widehat{u_0}(\xi_1)\,
      |\xi_2|^\beta \overline{\widehat{u_0}(\xi_2)}\,
      |\xi_3|^\beta \widehat{u_0}(\xi_3)\,d\xi_1\,d\xi_2,
\end{split}
\end{equation}
where we use the shorthand
\[
 \vec{\xi} = (\xi,\xi_1,\xi_2,\xi_3),\qquad
 \Omega(\vec{\xi}) := \omega(\xi) - \omega(\xi_1) + \omega(\xi_2) - \omega(\xi_3).
\]
The factor
\[
 \frac{e^{it \Omega(\vec{\xi})}-1}{\Omega(\vec{\xi})}
\]
is harmless whenever $|\Omega(\vec{\xi})|$ is not too small: for fixed $t$ it behaves like $t$ when
$\Omega(\vec{\xi})$ is small and like $1/\Omega(\vec{\xi})$ when $\Omega(\vec{\xi})$ is large. In each of our
counterexamples, we will choose frequencies so that $\Omega(\vec{\xi})$ is small but non-zero and
can be approximated by a quantity whose size we can control precisely.

The $H^s$-norm of $d_0^3\Phi_t(u_0,u_0,u_0)$ is thus
\begin{equation}\label{Hs-norm-d3Phi}
 \bigl\|d_0^3\Phi_t(u_0,u_0,u_0)\bigr\|_{H^s}^2
  = \int_{\R} \langle \xi\rangle^{2s}\,
     \bigl|\widehat{d_0^3\Phi_t(u_0,u_0,u_0)}(\xi)\bigr|^2\,d\xi.
\end{equation}
In the constructions below, we will take $u_0$ as a superposition of functions whose Fourier
transforms are indicator functions of short intervals in frequency space. This allows us to
evaluate~\eqref{Hs-norm-d3Phi} explicitly up to multiplicative constants and to detect
growth in the parameter $N$ that contradicts~\eqref{C3-bound}.

\subsection{Optimality of \texorpdfstring{$s\ge 2\beta + \frac{2-\alpha}{4}$}{s≥2β+(2−α)/4}}\label{subsec:opt-s}

We first show that the lower bound
\[
 s \;\ge\; 2\beta + \frac{2-\alpha}{4}
\]
is necessary for the existence of a $\mathcal C^3$ flow map at the origin. This corresponds
to the high--high-high to high interaction.

Let $N\in 2^{\N}$ be a large dyadic frequency and $0<\lambda\ll N$ a small parameter to be chosen
as a negative power of $N$. Define the intervals
\[
 I_1 := [N, N+\lambda],\qquad
 I_2 := [N-4\lambda, N-3\lambda],
\]
and set
\[
 \widehat{\phi_{1,N}}(\xi) := \widehat{\phi_{3,N}}(\xi)
   := \chi_{I_1}(\xi)\, N^{-s}\lambda^{-1/2},
 \qquad
 \widehat{\phi_{2,N}}(\xi) := \chi_{I_2}(\xi)\, N^{-s}\lambda^{-1/2}.
\]
We then take
\[
 u_0 := \phi_{1,N} + \phi_{2,N} + \phi_{3,N}.
\]
By construction,
\begin{align*}
 \|\phi_{j,N}\|_{H^s}^2
   &= \int_{\R} \langle \xi\rangle^{2s} |\widehat{\phi_{j,N}}(\xi)|^2\,d\xi
    \sim \langle N\rangle^{2s} N^{-2s}\lambda^{-1}\lambda
    \sim 1,
\end{align*}
uniformly in $N$, for $j=1,2,3$. Hence
\begin{equation}\label{u0-Hs-normalised}
 \|u_0\|_{H^s} \sim 1.
\end{equation}

If $\xi_1,\xi_3\in I_1$ and $\xi_2\in I_2$, then
\[
 \xi_1-\xi_2\in [3\lambda,5\lambda], \qquad
 \xi_3-\xi_2\in [3\lambda,5\lambda],
\]
and $\xi := \xi_1-\xi_2+\xi_3$ lies in an interval of length $\sim \lambda$ centred at frequency
$\sim N$. More precisely,
\[
 \xi \in [N+3\lambda, N+6\lambda].
\]
Thus the interaction
\[
 S(t_1)D_x^\beta \phi_{1,N}\,
 \overline{S(t_1)D_x^\beta \phi_{2,N}}\,
 S(t_1)D_x^\beta \phi_{3,N}
\]
produces an output Fourier support contained in $[N+3\lambda, N+6\lambda]$, and this support
is disjoint from the Fourier supports of all other possible cubic interactions formed from
$\phi_{1,N},\phi_{2,N},\phi_{3,N}$. Consequently, in the computation of $d_0^3\Phi_t(u_0,u_0,u_0)$, 
we may isolate the contribution of this specific interaction and ignore all others.

On the set $\xi_1,\xi_3\in I_1$, $\xi_2\in I_2$, we have
\[
 |\xi|\sim N,\qquad |\xi_j|\sim N,\quad j=1,2,3,
\]
and $|\xi_1-\xi_2|\sim|\xi_3-\xi_2|\sim \lambda\ll N$. A Taylor  theorem of $\omega(\xi)=|\xi|^\alpha$
around $\xi\sim N$ yields
\[
\begin{split}
 \omega(\xi_1) & = \omega(\xi) + \omega'(\xi)(\xi_1-\xi) + \frac12 (\xi_1-\xi)^2 \int_0^1 \omega'' (\xi + \theta (\xi_1 - \xi)) d\theta\\
 & = \omega(\xi) + \omega'(\xi)(\xi_1-\xi) + O(\lambda^2 N^{\alpha-2}),
\end{split}
\]
and similarly for $\omega(\xi_2)$ and $\omega(\xi_3)$. Using that
$ \xi_1-\xi_2+\xi_3 = \xi$,
one checks that the first-order terms in $(\xi_j-\xi)$ cancel in the expression $\Omega(\vec{\xi}) = \omega(\xi)-\omega(\xi_1)+\omega(\xi_2)-\omega(\xi_3)$,
and one is left with a second-order contribution of the form
\[
 |\Omega(\vec{\xi})| \sim \lambda^2 N^{\alpha-2}.
\]
Choosing
\[
 \lambda = N^{\frac{2-\alpha}{2}-\varepsilon} \ll N,\qquad \varepsilon>0 \text{ small},
\]
we obtain
\[
 |\Omega(\vec{\xi})| \sim N^{-2\varepsilon},
\]
uniformly in the interaction region.

For fixed $t>0$ and $N \gg 1$, we therefore have that the quantity
\[
 \left|\frac{e^{it \Omega(\vec{\xi})}-1}{\Omega(\vec{\xi})}\right| \sim |t|,
\]
up to an error of order $N^{-2\varepsilon}$ which can be absorbed
for large $N$.

Restricting~\eqref{Est-F1} to the specific interaction above and using the previous estimates,
we find, for $\xi$ in the output interval $[N+3\lambda,N+6\lambda]$,
\[
 \bigl|\mathcal F_x ( {d_0^3\Phi_t(u_0,u_0,u_0)})(\xi)\bigr|
  \gtrsim |t|\,
     |\xi|^\beta \, N^{3\beta}\, N^{-3s}\,\lambda^{-3/2}\, \lambda^2
 \sim |t|\, N^{4\beta-3s}\,\lambda^{1/2},
\]
where the factor $\lambda^2$ comes from the measure of the integration region in $(\xi_1,\xi_2)$,
and $\lambda^{-3/2}$ from the three factors $\lambda^{-1/2}$ in the amplitudes of
$\widehat{\phi_{j,N}}$. Since $\xi\sim N$ on this set, we also have $\langle\xi\rangle^s\sim N^s$.
Hence
\[
 \|d_0^3\Phi_t(u_0,u_0,u_0)\|_{H^s}^2
  \gtrsim |t|^2 \int_{N+3\lambda}^{N+6\lambda} N^{2s}\, N^{8\beta-6s}\,\lambda\, d\xi
 \sim |t|^2 N^{8\beta-4s}\,\lambda^2.
\]
Recalling that $\lambda = N^{\frac{2-\alpha}{2}-\varepsilon}$, we finally obtain
\begin{equation}\label{CounterEx2Norm}
 \|d_0^3\Phi_t(u_0,u_0,u_0)\|_{H^s}
  \gtrsim |t|\, N^{4\beta + \frac{2-\alpha}{2} - 2s - \varepsilon}.
\end{equation}

Since $\|u_0\|_{H^s}\sim 1$ by~\eqref{u0-Hs-normalised}, the $\mathcal C^3$ bound~\eqref{C3-bound} would
force the right-hand side of~\eqref{CounterEx2Norm} to remain bounded as $N\to\infty$.
This is impossible when
\[
 4\beta + \frac{2-\alpha}{2} - 2s >0
 \quad\Longleftrightarrow\quad
 s < 2\beta + \frac{2-\alpha}{4}.
\]
We conclude that if $s < 2\beta + \frac{2-\alpha}{4}$, the map $\Phi_t$ cannot be $\mathcal C^3$ at $0$,
which proves the optimality of the regularity threshold in Theorem~\ref{Th-ill-Line}.

\subsection{Optimality of \texorpdfstring{$\beta\le(\alpha-1)/2$}{β ≤ (α−1)/2}}\label{subsec:opt-beta-upper}

We now show that the condition
\[
 \beta \le \frac{\alpha-1}{2}
\]
is necessary. This corresponds to a high--low--low to high frequency interaction.

Let $N\in 2^{\N}$ be large and $0<\lambda\ll 1$ be a small parameter.
Define the intervals
\[
 I_1 := [N, N+\lambda],\qquad
 I_2 := [1+\lambda, 1+2\lambda],\qquad
 I_3 := [1+4\lambda, 1+5\lambda],
\]
and set
\[
 \widehat{\phi_{1,N}}(\xi) := \chi_{I_1}(\xi)\, N^{-s}\lambda^{-1/2},\qquad
 \widehat{\phi_{2,N}}(\xi) := \chi_{I_2}(\xi)\,\lambda^{-1/2},
 \qquad
 \widehat{\phi_{3,N}}(\xi) := \chi_{I_3}(\xi)\,\lambda^{-1/2}.
\]
Thus $\phi_{2,N}$ and $\phi_{3,N}$ are supported at frequencies of size $\sim 1$, while
$\phi_{1,N}$ is supported near $N$. As before, one checks easily that
\[
 \|\phi_{j,N}\|_{H^s} \sim 1,\qquad j=1,2,3,
\]
so that $u_0:=\phi_{1,N}+\phi_{2,N}+\phi_{3,N}$ satisfies $\|u_0\|_{H^s}\sim 1$.

If $\xi_1\in I_1$, $\xi_2\in I_2$ and $\xi_3\in I_3$, then
\[
 \xi_3-\xi_2 \in [2\lambda, 4\lambda], \quad \textup{ and } \quad
 \xi := \xi_1-\xi_2+\xi_3 \in [N+2\lambda, N+5\lambda].
\]
Thus the output frequency is again $\xi\sim N$ and the region of interaction in $(\xi_1,\xi_2,\xi_3)$
has measure $\sim \lambda^3$. A mean value theorem  
\[
|\xi_1|^\alpha - |\xi|^\alpha  = \al (\xi_1 - \xi) \int_0^1 |\xi + \theta(\xi_1 - \xi)|^{\al - 1} d\theta
\] 
yields
\[
 |\Omega(\vec{\xi})| \sim \lambda N^{\alpha-1}.
\]
Choosing $\lambda := N^{1-\alpha-\varepsilon}$, 
we have $|\Omega(\vec{\xi})|\sim N^{-\varepsilon}$. For fixed $t>0$, it follows again that
\[
 \left|\frac{e^{it \Omega(\vec{\xi})}-1}{\Omega(\vec{\xi})}\right|
 \sim |t|
\]
for large $N$.

Inserting this configuration into~\eqref{Est-F1} and arguing as before, we obtain on
$\xi\in[N+2\lambda,N+5\lambda]$,
\[
 \bigl|\mathcal F_x ( {d_0^3\Phi_t(u_0,u_0,u_0)}) (\xi) \bigr|
  \gtrsim |t|\, N^\beta \cdot N^\beta \cdot 1 \cdot \;
          N^{-s}\lambda^{-3/2}\,\lambda^2
  \sim |t|\, N^{2\beta-s}\lambda^{1/2},
\]
and hence
\begin{equation}\label{CounterEx3Norm}
 \|d_0^3\Phi_t(u_0,u_0,u_0)\|_{H^s}
  \gtrsim   |t|\, N^{2\beta}\lambda \sim |t|\, N^{1-\alpha + 2\beta -\varepsilon},
\end{equation}
with $\lambda = N^{1-\alpha-\varepsilon}$.
If $\beta > (\alpha-1)/2$, then the exponent on $N$ in~\eqref{CounterEx3Norm} is positive,
so the right-hand side diverges as $N\to\infty$, while $\|u_0\|_{H^s}\sim 1$ remains bounded.
This contradicts the $\mathcal C^3$ bound~\eqref{C3-bound}, and therefore the flow map cannot be $\mathcal C^3$
at the origin when $\beta>(\alpha-1)/2$. This proves the optimality of the upper bound
$\beta\le (\alpha-1)/2$.

\subsection{Optimality of \texorpdfstring{$\beta\ge -1/4$}{β ≥ −1/4}}\label{subsec:opt-beta-lower}

Finally, we show that $\beta\ge -\tfrac14$ is necessary. Here the relevant interaction is
low--low--low to low, in which the frequencies remain close to zero.

Let $N\in 2^{\N}$ be large and set
 $\lambda := N^{-1}$.
Define the interval
 $I := [2\lambda,3\lambda]$,
and let
\[
 \widehat{\phi_{1,N}}(\xi) = \widehat{\phi_{2,N}}(\xi) = \widehat{\phi_{3,N}}(\xi)
  := \chi_{I}(\xi)\,\lambda^{-1/2}.
\]
Thus each $\phi_{j,N}$ is supported in a small interval near the origin of length $\lambda$,
and its Fourier amplitude is of size $\lambda^{-1/2}$. A simple computation shows that
\[
 \|\phi_{j,N}\|_{H^s}^2
   = \int_I \langle\xi\rangle^{2s} \lambda^{-1}\,d\xi
   \sim \lambda^{-1}\lambda \sim 1,
\]
uniformly in $N$, since $|\xi|\lesssim\lambda\ll1$ on $I$. Hence again $\|u_0\|_{H^s}\sim1$ for
$u_0:=\phi_{1,N}+\phi_{2,N}+\phi_{3,N}$.

For $\xi_1,\xi_2,\xi_3\in I$, we have $|\xi_j|\sim\lambda$ and
\[
 \xi := \xi_1-\xi_2+\xi_3 \in [\lambda,4\lambda],
\]
so that the output frequency also satisfies $|\xi|\sim \lambda$.
Using the homogeneity of order $\alpha>1$ of $\omega(\xi)=|\xi|^{\alpha}$, we find $|\Omega(\vec{\xi})|
  = \bigl||\xi|^{\alpha}-|\xi_1|^{\alpha}+|\xi_2|^{\alpha}-|\xi_3|^{\alpha}\bigr|
  \lesssim \lambda^{\alpha}
  = N^{-\alpha}$.
In particular, for fixed $t>0$, the factor
\[
\left| \frac{e^{it\Omega(\vec{\xi})}-1}{\Omega(\vec{\xi})}\right| \sim |t|
\]
for large $N$.

On the interaction set, we have $|\xi|\sim\lambda$, $|\xi_j|\sim\lambda$ and
\[
 |\xi|^\beta \sim \lambda^\beta,\qquad |\xi_j|^\beta \sim \lambda^\beta.
\]
The support of the interaction in $(\xi_1,\xi_2)$ has measure $\sim \lambda^2$ and the amplitudes
contribute $\lambda^{-3/2}$. Hence from~\eqref{Est-F1} we obtain
\[
 \bigl| \mathcal F_x ({d_0^3\Phi_t(u_0,u_0,u_0)})(\xi)\bigr|
  \gtrsim |t|\,
    \lambda^\beta \cdot \lambda^{3\beta} \cdot \lambda^{-3/2} \cdot \lambda^2
  = |t|\,\lambda^{4\beta + \frac12},
\]
for $\xi$ in an interval of length $\sim\lambda$ with $|\xi|\sim\lambda$.
Since $\langle\xi\rangle^s\sim 1$ for such small $\xi$, we find
\[
 \|d_0^3\Phi_t(u_0,u_0,u_0)\|_{H^s}^2
  \gtrsim |t|^2 \int_{\lambda}^{4\lambda} \lambda^{2(4\beta+1/2)}\,d\xi
  \sim |t|^2 \lambda^{8\beta+1}\,\lambda,
\]
and so
\begin{equation}\label{CounterEx4Norm}
 \|d_0^3\Phi_t(u_0,u_0,u_0)\|_{H^s}
  \gtrsim |t|\,\lambda^{4\beta+1}
  = |t|\,N^{-(4\beta+1)}.
\end{equation}
If $\beta<-1/4$, then $4\beta+1<0$ and the right-hand side of~\eqref{CounterEx4Norm}
diverges as $N\to\infty$, whereas $\|u_0\|_{H^s}\sim1$ remains bounded. This again contradicts
the $\mathcal C^3$ bound~\eqref{C3-bound} and shows that one must require $\beta\ge -1/4$.

\subsection{Proof of Theorem~\ref{Th-ill-Line}}

Combining the three counterexamples above, we obtain:

\begin{itemize}
 \item[(i)] If $s < 2\beta + \frac{2-\alpha}{4}$, then the high--high--high to high example
   in Subsection~\ref{subsec:opt-s} shows that $d_0^3\Phi_t$ cannot satisfy~\eqref{C3-bound}.
 \item[(ii)] If $\beta > (\alpha-1)/2$, then the high--low--low to high example in
   Subsection~\ref{subsec:opt-beta-upper} violates~\eqref{C3-bound}.
 \item[(iii)] If $\beta < -1/4$, then the low--low--low to low example in
   Subsection~\ref{subsec:opt-beta-lower} also violates~\eqref{C3-bound}.
\end{itemize}

In each of these cases, the solution map $\Phi_t$ cannot be $\mathcal C^3$ at the origin from $H^s(\R)$
to itself. This proves Theorem~\ref{Th-ill-Line}.
\qedhere

\medskip

The first counterexample shows that, for $\beta>\frac18-\frac{\alpha}{4}$, the regularity
requirement
\[
 s \;\ge\; 2\beta + \frac{2-\alpha}{4},
\]
obtained in Remark~\ref{RMK:LWP}, is sharp. For $\beta\le\frac18-\frac{\alpha}{4}$,
the well-posedness threshold in Remark~\ref{RMK:LWP} is strictly above the scaling index,
but our current approach does not decide whether it is optimal. This is related to the need
to treat negative regularity exponents by a direct trilinear estimate rather than reducing
to the bilinear estimate via the simple inequality~\eqref{est-i1}.

The third counterexample shows that the contraction mapping argument in $H^s$ cannot
cover the case $\beta<-1/4$, but does not rule out the possibility of establishing local
well-posedness at the endpoint $\beta=-1/4$ by different methods. This endpoint corresponds,
in physical applications, to the turbulent regime of the original MMT model~\cite{MMT-97},
and it would be very interesting to determine whether a well-posedness theory can be developed
there.



\bigskip
\subsection*{Acknowledgment} The first author would like to thank FAPESP Brazil for financial support under grant (\#2024/10613-4) and the School of Mathematics, University of Birmingham, UK, for hospitality where this work was developed. Y.W. was supported by the EPSRC Mathematical Sciences Small Grant (grant no. UKRI1116).\\


\noindent
{\bf Conflict of interest statement.} 
On behalf of all authors, the corresponding author states that there is no conflict of interest.\\

\noindent 
{\bf Data availability statement.} 
The datasets generated during and/or analyzed during the current study are available from the corresponding author on reasonable request.


\end{document}